\documentclass[a4paper,10pt,twoside,english]{amsart}

\usepackage[final]{microtype}

\usepackage[utf8]{inputenc}
\usepackage{amsthm}
\usepackage{amsmath}
\usepackage{eucal}
\usepackage{enumitem}
\usepackage{tikz}
\usepackage[sort,nocompress]{cite}
\usetikzlibrary{matrix,arrows,decorations.pathmorphing,calc,cd,shapes,fit,decorations.pathreplacing,shapes.geometric}
\usepackage{cancel}
\usepackage{tikz-cd}
\usepackage{tensor}
\usepackage{enumitem}
\usepackage{mathtools, mathdots}
\usepackage{amssymb,graphicx}
\usetikzlibrary{shapes}
\usepackage{diagbox}

\makeatletter
\newcommand{\thickhline}{%
    \noalign {\ifnum 0=`}\fi \hrule height 1pt
    \futurelet \reserved@a \@xhline
}
\newcolumntype{"}{@{\hskip\tabcolsep\vrule width 1pt\hskip\tabcolsep}}
\makeatother

\tikzset{
   dashellipse/.style={ellipse,draw,dashed,inner xsep=-7pt,blue,fit={#1}}
}

\DeclareMathOperator{\add}{\mathsf{add}}

\DeclareMathOperator{\Hom}{\mathsf{Hom}}
\DeclareMathOperator{\End}{\mathsf{End}}

\DeclareMathOperator{\Ext}{\mathsf{Ext}}

\DeclareMathOperator{\Db}{\mathsf D^{\mathsf b}}

\DeclareMathOperator{\rad}{\mathsf{rad}}

\DeclareMathOperator{\minimum}{\mathsf{min}}
\DeclareMathOperator{\maximum}{\mathsf{max}}

\usepackage[utf8]{inputenc}
\usepackage{amsmath}
\usepackage{enumitem}
\usepackage[sort,nocompress]{cite}
\usepackage{tikz-cd}

\renewcommand{\to}{\longrightarrow}

\newcommand{\xto}{\xrightarrow}

\numberwithin{equation}{section}

\newtheorem{lemma}[equation]{Lemma}
\newtheorem{corollary}[equation]{Corollary}
\newtheorem{proposition}[equation]{Proposition}

\newtheorem*{theorem*}{Theorem}

\newtheorem*{proposition*}{Proposition}
\newtheorem*{definition*}{Definition}
\newtheorem*{example*}{Example}
\newtheorem*{lemma*}{Lemma}
\newtheorem*{remark*}{Remark}
\theoremstyle{definition}
\newtheorem*{exampleA11/rad5revisited}{Example~\ref{example:A11/rad5} (revisited)}
\newtheorem*{exampleA11/rad4revisited}{Example~\ref{example:A11/rad4} (revisited)}
\newtheorem*{examplesimplemindedrevisited}{Example~\ref{example:simpleminded} (revisited)}

\newtheorem{example}[equation]{Example}
\newtheorem{construction}[equation]{Construction}
\newtheorem*{construction*}{Construction}
\newtheorem{nonexample}[equation]{Non-example}
\newtheorem{nonexample*}{Non-example}

\newtheorem*{caveat*}{Caveat}
\newtheorem{liste}[equation]{List}
\newtheorem*{liste*}{List}
\newtheorem{remark}[equation]{Remark}

\usepackage[colorinlistoftodos]{todonotes}
\usepackage{hyperref}

\title{Non-piecewise hereditary Nakayama algebras}
\author[Fosse, Oppermann and Stai]{Didrik Fosse, Steffen Oppermann and Torkil Stai}
\address{Department of Mathematical Sciences, NTNU, 7491 Trondheim, Norway}
\email{steffen.oppermann@ntnu.no}
\email{torkil.stai@ntnu.no}
\email{didrikfosse@hotmail.com}
\thanks{The second named author  was supported by the Centre for Advanced Study (CAS) in Oslo, Norway, within the research project ``Representation theory'' during the 2022/23 academic year. The third named author was supported by Norges forskningsr{\aa}d, grant 302223.}
\begin{document}
\maketitle
\begin{abstract}
Happel and Seidel gave a classification of piecewise hereditary Nakayama algebras, where the relations are given by some power of the radical.

Here we explore what happens for general relations. We develop techniques for showing that a given algebra is not piecewise hereditary, illustrating them on numerous mid-sized examples.  Then we observe cases where the property of being non-piecewise hereditary can be extended to other (larger) Nakayama algebras.
While a complete classification remains elusive, we are able to identify two types of patterns of relations preventing piecewise heredity, indicating that for large quivers many Nakayama algebras are non-piecewise hereditary.

\end{abstract}
\section{Introduction}

Since their introduction in \cite{Nakayama1940}, Nakayama algebras have been the subject of extensive study. While their representation theory is well understood (see e.g.\ \cite{ASS2006}), there are still a number of open questions regarding their homological nature. In particular, the general understanding of derived categories of Nakayama algebras is lacking.

While developing a complete classification of Nakayama algebras up to derived equivalence is out of the scope of this paper, a related (and hopefully simpler) problem is determining whether a given Nakayama algebra is piecewise hereditary. An algebra $\Lambda$ is \textbf{piecewise hereditary} if there is a hereditary abelian category $\mathcal H$ and an equivalence of triangulated categories
\begin{equation}\label{equation:pwhdef}
\Db(\Lambda) \simeq \Db(\mathcal H). \tag{$\ast$}
\end{equation}

A Nakayama algebra is either \textbf{cyclic}, that is its ordinary quiver is an oriented cycle, or \textbf{linear}, that is its ordinary quiver is linearly ordered \( \mathbb A_n \). Since all piecewise hereditary algebras are triangular (see \cite[Theorem~1.1(i)]{Happel-Reiten-Smalo}), we focus our attention on the linear case. More precisely, take a field $k$ and an admissible ideal $J$ in the path algebra $k\mathbb A_n$. We are interested in the following question. \[\textit{When is the algebra $k\mathbb A_n/J$ piecewise hereditary?}\]

Happel and Seidel \cite[Table~1]{MR2736030} gave a complete solution in the case where $J$ is a power of the radical and $k$ is algebraically closed. Beyond this setup it seems that little is known, maybe surprisingly given how well we know the module categories of Nakayama algebras.

\bigskip
Concurrent to this paper, the first named author showed that linear Nakayama algebras whose relations overlap by at most one arrow are always piecewise hereditary \cite{Fosse2023}. Research for both these papers started with his combinatorial tilting mutation rule \cite{Fosse2021}, and computer experiments based on that rule.

The number of algebras of the form $k\mathbb A_n/J$ equals the $(n-1)$-th Catalan number, see e.g.\ \cite{MR3049570}. Exhaustive searches show each such algebra is piecewise hereditary as long as $n\le 8$. Among the 1430 quotients of $k\mathbb A_9$ there is precisely one algebra which is not piecewise hereditary (see Example~\ref{example:A9}), namely

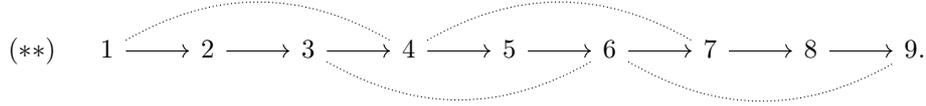
\begin{equation}\label{equation:introquiver}
\begin{tikzcd}
1 \ar[r] \ar[rrr,bend left,densely dotted,-] & 2 \ar[r] & 3 \ar[r] \ar[rrr,bend right,densely dotted,-] & 4 \ar[r] \ar[rrr,bend left,densely dotted,-] & 5 \ar[r] & 6 \ar[r] \ar[rrr,bend right,densely dotted,-] & 7 \ar[r] & 8 \ar[r] & 9. \tag{$\ast\ast$}
\end{tikzcd}
\end{equation}

This might give the impression that ``most" Nakayama algebras are piecewise hereditary. However, as we will argue here, this is only an artefact due to the smallness of the $n$ for which computer experiments are feasible. In fact we will see that as $n$ grows there is plenty of opportunity for $k\mathbb A_n/J$ to \emph{not} be piecewise hereditary.

\subsection*{Summary of results}
Inspired by---but not actually using---Chen and Ringel's \cite{MR3896230}, we establish two methods for showing that a given linear Nakayama algebra is not piecewise hereditary: In Proposition~\ref{prop.coarse_fine} we use a rather explicit construction of a complex, while in Lemma~\ref{lemma:taupathimpliesnotpwh} we use Auslander--Reiten translation in the derived category. Both let us establish a handful of concrete quotients of $k\mathbb A_n$ (with reasonably small $n$) which are not piecewise hereditary. Some of these algebras also serve as the starting point for more general patterns observed later.

One way to extend our examples is via derived equivalence: In Section~\ref{sect:derived_equiv} we give two general types of derived equivalences between Nakayama algebras. Firstly, in Corollary~\ref{corollary:lengthtworelations} we see that relations of length two do not affect the derived equivalence type. Secondly, in Proposition~\ref{proposition:doubleMutation} we observe that certain tilting mutations of Nakayama algebras give new Nakayama algebras.

Beyond derived equivalences, Corollary~\ref{corollary:introducevertex} provides recipes for adding vertices (and appropriate relations) which preserve the property of being non-piecewise hereditary. Application of these results to our initial examples helps pinpoint certain traits of $J$ which imply that $k\mathbb A_n/J$ is not piecewise hereditary, and we investigate such patterns of relations in two directions in particular:

Proposition~\ref{proposition:A13tworelations} shows that a Nakayama algebra is not piecewise hereditary if it has a pair of relations whose overlap is at least six arrows, with at least three arrows between their starts, and at least three arrows between their ends. The simplest such algebra is
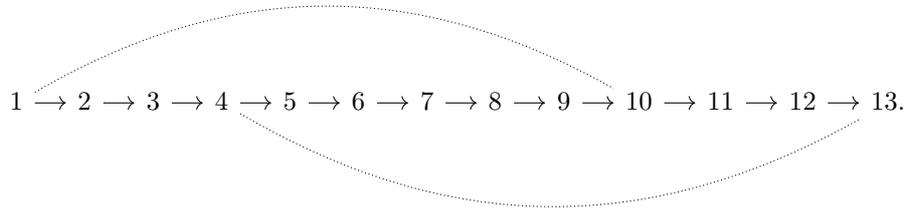
\begin{equation*}
\begin{tikzcd}[column sep = small]
1 \ar[r] \ar[rrrrrrrrr, bend left, densely dotted, -] & 2 \ar[r] & 3 \ar[r] & 4 \ar[r] \ar[rrrrrrrrr, bend right, densely dotted, -]& 5 \ar[r] & 6 \ar[r] & 7 \ar[r] & 8 \ar[r] & 9 \ar[r] & 10 \ar[r] & 11 \ar[r] & 12 \ar[r] & 13.
\end{tikzcd}
\end{equation*}

Similarly, Proposition~\ref{proposition:A9extended} essentially says that if a Nakayama algebra has a pair of relations with an overlap of two or more arrows, with a relation on each side of the pair, then the algebra is not piecewise hereditary. Here the archetypical picture is the one given in \eqref{equation:introquiver} above.

\bigskip
Finally, note that in \cite{MR2736030} the assumption $k=\overline k$ let the authors invoke Happel's characterization of hereditary abelian categories \cite{MR1827736} i.e.,\ a \emph{list} of those $\mathcal H$ that can occur in the equivalence (\ref{equation:pwhdef}): Such an $\mathcal H$ will be either the module category of a hereditary algebra or derived equivalent to a canonical algebra. So in order to show that a given algebra is not piecewise hereditary, it sufficed to observe that it does not have the same Coxeter polynomial as any algebra from the \emph{list}.

There is no guarantee that such an approach will work, however, as Coxeter polynomials do not \emph{characterize} derived equivalence classes. Indeed, also among the quotients of $k\mathbb A_n$ there are algebras which share a common Coxeter polynomial where one is piecewise heredetary and the other is not (see Remark~\ref{remark:Coxeter}). Without appealing to the \emph{list} or to Coxeter polynomials, or requiring algebraic closedness of the base field, in Proposition~\ref{proposition:HS} we recover the part of Happel--Seidel's Table~1 which consists of non-piecewise hereditary Nakayama algebras.

\subsection*{Notation and conventions}

We will focus on linear Nakayama algebras, that is algebras of the form \( k \mathbb{A}_n / J \). Thus we always implicitly assume that the Nakayama algebras we encounter are linear, and that their vertices are numbered from $1$ to $n$. When referring to the relations of a Nakayama algebra we mean a minimal set of relations.

We consider covariant representations, so the indecomposable projective module \( P_i \) has as $k$-basis all paths starting in vertex \( i \). In particular there will always be non-zero maps \( P_i \to P_{i-1} \), but not the other way.

\section{Derived equivalences between Nakayama algebras} \label{sect:derived_equiv}

\subsection*{Relations of length two}

The following observation shows that for the purpose of determining derived equivalence, relations of length two may safely be ignored:

\begin{corollary}\label{corollary:lengthtworelations}
Let $\Lambda$ and $\Gamma$ be Nakayama algebras such that one can be obtained from the other by adding and removing length-two relations.

Then $\Lambda$ and $\Gamma$ are derived equivalent.
\end{corollary}

We dub this result a corollary since it is an application of the following more general fact.

\begin{proposition}\label{proposition:lengthtworelations}
Let $A$ and $B$ be $k$-algebras with the global dimension of $B$ finite, let $M$ be a left $A$-module, and let $N$ be a right $B$-module. Then the algebras \[\text{$\Lambda = \left(\begin{matrix}A&0&0\\M&k&0\\0&N&B\end{matrix}\right)$ and $\Gamma= \left(\begin{matrix}A&0&0\\M&k&0\\M\otimes_k N&N&B\end{matrix}\right)$}\] are derived equivalent.
\end{proposition}

In terms of quivers and relations, $\Lambda$ corresponds to something like
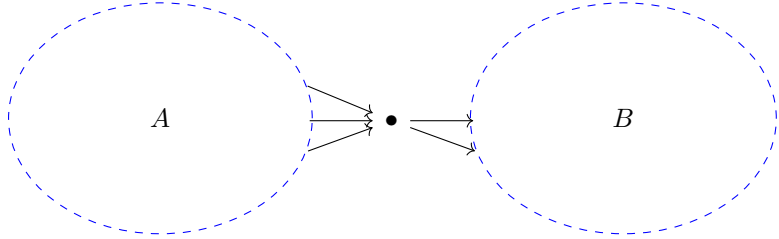
\[
\tikzset{
F/.style = {ellipse, draw=blue, dashed,
        inner xsep=-2mm,inner ysep=-4mm, rotate=0,
        fit=#1}
    }
\begin{tikzcd}[row sep=.2cm,
            arrows = dash,
execute at end picture = {
\node[F = (\tikzcdmatrixname-1-1) (\tikzcdmatrixname-1-3) (\tikzcdmatrixname-5-3) (\tikzcdmatrixname-5-1)] {};
\node[F = (\tikzcdmatrixname-1-5) (\tikzcdmatrixname-1-7) (\tikzcdmatrixname-5-5) (\tikzcdmatrixname-5-7)] {};
                    }
                ]
\phantom{\bullet}& \phantom{\bullet}& \phantom{\bullet}&[+10pt] \phantom{\bullet}&[+10pt] \phantom{\bullet}& \phantom{\bullet}& \phantom{\bullet} \\
\phantom{\bullet}& \phantom{\bullet}& \phantom{\bullet}\ar[dr,->,shorten <= 1em]& \phantom{\bullet}& \phantom{\bullet}& \phantom{\bullet}& \phantom{\bullet} \\
\phantom{\bullet}& A& \phantom{\bullet} \ar[r,->,shorten <= 1em]& \bullet \ar[r,->,shorten >= 1em] \ar[dr,->,shorten >= 1em] & \phantom{\bullet}& B& \phantom{\bullet} \\
\phantom{\bullet}& \phantom{\bullet}& \phantom{\bullet} \ar[ur,->,shorten <= 1em]& \phantom{\bullet}& \phantom{\bullet}& \phantom{\bullet}& \phantom{\bullet} \\
\phantom{\bullet}& \phantom{\bullet} & \phantom{\bullet} & \phantom{\bullet}& \phantom{\bullet}& \phantom{\bullet}& \phantom{\bullet} \\
\end{tikzcd}
\]
with each path from $A$ to $B$ being a relation, while $\Gamma$ is the same quiver but with \emph{no} relations from $A$ to $B$. Clearly, Proposition~\ref{proposition:lengthtworelations} implies the claim in Corollary~\ref{corollary:lengthtworelations}.

\begin{proof}[Proof of Proposition~\ref{proposition:lengthtworelations}]
Consider the $\Lambda$-module
\[T=\underset{\substack{\uparrow\\\mathclap{T_1}}}{\left(\begin{matrix}A\\M\\0\end{matrix}\right)} \oplus
\underset{\substack{\uparrow\\\mathclap{T_2}}}{\left(\begin{matrix}0\\k\\0\end{matrix}\right)} \oplus \underset{\substack{\uparrow\\\mathclap{T_3}}}{\left(\begin{matrix}0\\DN\\DB\end{matrix}\right)} .\]
Since
\begin{align*}
\End_{\Lambda}(T) & = \left(\begin{matrix}\Hom_{\Lambda}(T_1,T_1) & \Hom_{\Lambda}(T_1,T_2) & \Hom_{\Lambda}(T_1,T_3) \\ \Hom_{\Lambda}(T_2,T_1) & \Hom_{\Lambda}(T_2,T_2) & \Hom_{\Lambda}(T_2,T_3) \\ \Hom_{\Lambda}(T_3,T_1) & \Hom_{\Lambda}(T_3,T_2) & \Hom_{\Lambda}(T_3,T_3) \end{matrix}\right) \\
 & = \left(\begin{matrix}A&0&0\\M&k&0\\\Hom_{k}(DN,M)&D^2N&\End_{\Lambda}(DB)\end{matrix}\right) \\
 & = \left(\begin{matrix}A&0&0\\M&k&0\\M\otimes_k N&N&B\end{matrix}\right)
= \Gamma
\end{align*}
it suffices to show that $T$ is a tilting module of finite projective dimension.

$T_1$ is projective. On the other hand, $T_2$ and $T_3$ are injectives over the algebra \[\left(\begin{matrix}k&0\\N&B\end{matrix}\right),\] so we can take projective resolutions over the latter. Since $B$ has finite global dimension, it follows that the projective dimension of $T$ is finite. Since there are no extensions from $T_2 \oplus T_3$ to $T_1$ we also have $\Ext_{\Lambda}^{\ge1}(T,T)=0$.

Dually, since projectives over the algebra \[\left(\begin{matrix}k&0\\N&B\end{matrix}\right)\] have finite injective dimension, we can take resolutions over this last algebra in order to produce finite coresolutions by $\add T$.
\end{proof}

\subsection*{Double tilting mutation}

While a single tilting mutation will seldom give rise to a derived equivalence between Nakayama algebras (and the cases where it does are in fact covered by our discussion of relations of length two in the previous subsection), we do get interesting derived equivalences by tilting-mutating twice. More precisely, we have the following result.

\begin{proposition}\label{proposition:doubleMutation}
Let \( \Lambda \) be a Nakayama algebra. Assume that \( r \) is a relation from vertex \( s \) to \( t \), such that there is another relation starting in vertex \( s-1 \), but no relation starting in \( t-1 \).

Then \( \Lambda \) is derived equivalent to the Nakayama algebra obtained from it by
\begin{itemize}
\item if \( t \) is not the final vertex, adding a new relation, of the same length as \( r \), from \( s+1 \) to \( t+1 \);
\item shortening at the start every relation starting properly within \( r \) (i.e.\ in a vertex that \( r \) passes through, not including \( s \) or \( t \));
\item lengthening toward its end any relation ending properly within \( r \) --- this will generally include the second relation from the assumption.

\end{itemize}
\end{proposition}

We call the new algebra obtained in this way \( L_t(\Lambda) \), since it is the result of left mutation. There is the dual notion of \( R_s(\Lambda) \). Note that left mutation tends to ``move relations to the right'', and vice versa.

The following illustrates the basic change to the relations in the assumption of Proposition~\ref{proposition:doubleMutation} from $\Lambda$ to $L_t(\Lambda)$.
If there are other relations in $\Lambda$ starting or ending properly within $r$, then their lengths will also be changed accordingly.

\medskip
$\Lambda\colon$
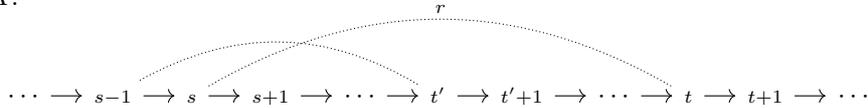
\vspace{-1em}
\begin{equation*}
\begin{tikzcd}[column sep = small]
\cdots \ar[r] & {\scriptstyle s-1} \ar[r] \ar[rrrr, bend left, densely dotted, -]& {\scriptstyle s} \ar[r] \ar[rrrrrr, bend left, densely dotted, -, "r"] & {\scriptstyle s+1} \ar[r] & \cdots \ar[r] & {\scriptstyle t'} \ar[r] & {\scriptstyle t'+1} \ar[r] & \cdots \ar[r] & {\scriptstyle t} \ar[r] & {\scriptstyle t+1} \ar[r] & \cdots
\end{tikzcd}
\end{equation*}

\medskip
$L_t(\Lambda)\colon$
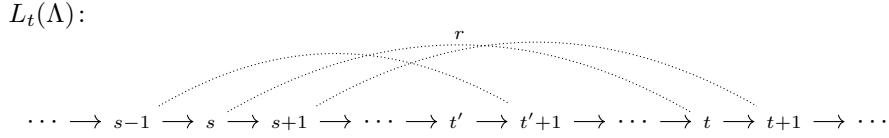
\vspace{-1em}
\begin{equation*}
\begin{tikzcd}[column sep = small]
\cdots \ar[r] & {\scriptstyle s-1} \ar[r] \ar[rrrrr, bend left, densely dotted, -]& {\scriptstyle s} \ar[r] \ar[rrrrrr, bend left, densely dotted, -, "r"] & {\scriptstyle s+1} \ar[r] \ar[rrrrrr, bend left, densely dotted, -] & \cdots \ar[r] & {\scriptstyle t'} \ar[r] & {\scriptstyle t'+1} \ar[r] & \cdots \ar[r] & {\scriptstyle t} \ar[r] & {\scriptstyle t+1} \ar[r] & \cdots
\end{tikzcd}
\end{equation*}

\begin{remark}
We found this result through computer experiments using \cite{Fosse2021}. However, after the fact, it was easy enough to give a direct proof.
\end{remark}

\begin{proof}[Proof of Proposition~\ref{proposition:doubleMutation}]
We apply left tilting mutation at \( t \) twice.

In the first step, note that by our assumption of not having relations starting in \( t-1 \) the natural map \( P_t \to P_{t-1} \) is mono. Its cokernel is the simple top of \( P_{t-1} \), which we denote by \( S_{t-1} \). By the relation \( r \), this simple is the socle of \( P_s \), but will not appear in the socle of any projective with a higher index. Therefore we have a monomorphic left approximation by projectives \( S_{t-1} \to P_s \), whose cokernel is \( P_s / S_{t-1} \).

It follows from the general theory of tilting mutation that
\[ \bigoplus_{i \neq t} P_i \oplus P_s / S_{t-1} \]
is a tilting module.

Note that \( P_s / S_{t-1} \) naturally maps to \( P_{s-1} \)  --- this because we assumed there to be a relation starting in \( s-1 \). Therefore \( P_s / S_{t-1} \) slots in between \( P_s \) and \( P_{s-1} \), and the quiver of the endomorphism ring of our tilting module is linear again.

Finally, let us track what relations appeared in this endomorphism ring:
First, clearly there is a relation from \( P_{t-1} \) to \( P_s / S_{t-1} \). Second, any relation from \( P_s \) will be replaced by a new relation from \( P_s / S_{t-1} \), and any relation ending in \( P_t \) will be replaced by a new relation ending in \( P_{t-1} \).
Finally, the length of any relation which passed through \( P_t \) has been reduced by one, and the length of any relation which passed through \( P_s \) has been increased by one.

To finish up, we adjust the vertex names to match the ascending numbering from before the mutation. For each \( i \) satisfying \( s\leq i \leq t-1\), the vertex corresponding to \( P_i \) is named \( i+1 \), and the new vertex corresponding to \( P_s/S_{t-1} \) is named \( s \).
\end{proof}

\begin{example}\label{example:mutationToA11_5}
Let $\Lambda$ be given by
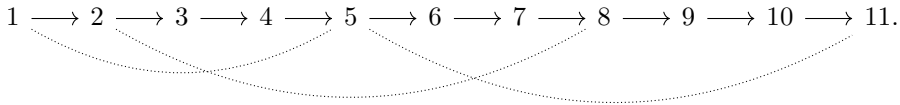
\begin{equation*}
\begin{tikzcd}[column sep = 1.8em]
1 \ar[r] \ar[rrrr, bend right, no head, densely dotted] & 2 \ar[r] \ar[rrrrrr, bend right, no head, densely dotted] &  3 \ar[r]  & 4 \ar[r] &  5 \ar[r] \ar[rrrrrr, bend right, no head, densely dotted] & 6 \ar[r] &  7 \ar[r] &  8 \ar[r] & 9 \ar[r] & 10 \ar[r] & 11.
\end{tikzcd}
\end{equation*}
By repeated and opportune application of $L_t$ and $R_s$ from Proposition~\ref{proposition:doubleMutation} we produce the following sequence of derived equivalent Nakayama algebras, which shows that there exists a derived equivalence between $\Lambda$ and $k \mathbb A_{11}/(\rad)^5$.

\medskip
$\Lambda_1 = L_8(\Lambda)\colon$
\begin{equation*}
\begin{tikzcd}[column sep = 1.8em]
1 \ar[r] \ar[rrrrr, bend right, no head, densely dotted] & 2 \ar[r] \ar[rrrrrr, bend right, no head, densely dotted] & 3 \ar[r] \ar[rrrrrr, bend right, no head, densely dotted] &  4 \ar[r] & 5 \ar[r] & 6 \ar[r] \ar[rrrrr, bend right, no head, densely dotted] &  7 \ar[r] & 8 \ar[r] & 9 \ar[r] & 10 \ar[r] & 11
\end{tikzcd}
\end{equation*}

$\Lambda_2 = L_9(\Lambda_1)\colon$
\begin{equation*}
\begin{tikzcd}[column sep = 1.8em]
1 \ar[r] \ar[rrrrrr, bend right, no head, densely dotted] & 2 \ar[r] & 3 \ar[r] \ar[rrrrrr, bend right, no head, densely dotted] &  4 \ar[r] \ar[rrrrrr, bend right, no head, densely dotted] & 5 \ar[r] & 6 \ar[r] &  7 \ar[r] \ar[rrrr, bend right, no head, densely dotted] & 8 \ar[r] & 9 \ar[r] & 10 \ar[r] & 11
\end{tikzcd}
\end{equation*}

$\Lambda_3 = L_{10}(\Lambda_2)\colon$
\begin{equation*}
\begin{tikzcd}[column sep = 1.8em]
1 \ar[r] \ar[rrrrrrr, bend right, no head, densely dotted] & 2 \ar[r] & 3 \ar[r] &  4 \ar[r]  \ar[rrrrrr, bend right, no head, densely dotted] & 5 \ar[r] & 6 \ar[r] &  7 \ar[r] & 8 \ar[r]  \ar[rrr, bend right, no head, densely dotted] & 9 \ar[r] & 10 \ar[r] & 11
\end{tikzcd}
\end{equation*}

$\Lambda_4 = L_8(\Lambda_3)\colon$
\begin{equation*}
\begin{tikzcd}[column sep = 1.8em]
1 \ar[r] \ar[rrrrrrr, bend right, no head, densely dotted] & 2 \ar[r] \ar[rrrrrrr, bend right, no head, densely dotted] & 3 \ar[r] &  4 \ar[r] & 5 \ar[r] \ar[rrrrr, bend right, no head, densely dotted] & 6 \ar[r] & 7 \ar[r] & 8 \ar[r]  \ar[rrr, bend right, no head, densely dotted] & 9 \ar[r] & 10 \ar[r] & 11
\end{tikzcd}
\end{equation*}

$\Lambda_5 = R_2(\Lambda_4)\colon$
\begin{equation*}
\begin{tikzcd}[column sep = 1.8em]
1 \ar[r] \ar[rrrrrr, bend right, no head, densely dotted] & 2 \ar[r]  \ar[rrrrrrr, bend right, no head, densely dotted] & 3 \ar[r] &  4 \ar[r]  \ar[rrrrrr, bend right, no head, densely dotted] & 5 \ar[r] & 6 \ar[r] & 7 \ar[r]  \ar[rrrr, bend right, no head, densely dotted]& 8 \ar[r] & 9 \ar[r] & 10 \ar[r] & 11
\end{tikzcd}
\end{equation*}

$\Lambda_6 = R_7(\Lambda_5)\colon$
\begin{equation*}
\begin{tikzcd}[column sep = 1.8em]
1 \ar[r] \ar[rrrrrr, bend right, no head, densely dotted] & 2 \ar[r] \ar[rrrrrr, bend right, no head, densely dotted] & 3 \ar[r] & 4 \ar[r] \ar[rrrrr, bend right, no head, densely dotted] & 5 \ar[r] & 6 \ar[r] \ar[rrrr, bend right, no head, densely dotted] &  7 \ar[r] \ar[rrrr, bend right, no head, densely dotted] & 8 \ar[r] & 9 \ar[r] & 10 \ar[r] & 11
\end{tikzcd}
\end{equation*}

$k \mathbb A_{11}/(\rad)^5 = R_{4}(\Lambda_6)  \colon$
\begin{equation*}
\begin{tikzcd}[column sep = 1.8em]
1 \ar[r] \ar[rrrrr, bend right, no head, densely dotted] & 2 \ar[r] \ar[rrrrr, bend right, no head, densely dotted] &  3 \ar[r] \ar[rrrrr, bend right, no head, densely dotted] &  4 \ar[r] \ar[rrrrr, bend right, no head, densely dotted] &  5 \ar[r] \ar[rrrrr, bend right, no head, densely dotted] & 6 \ar[r] \ar[rrrrr, bend right, no head, densely dotted] &  7 \ar[r] &  8 \ar[r] & 9 \ar[r] & 10 \ar[r] & 11
\end{tikzcd}
\end{equation*}
\end{example}

\begin{example}\label{example:A13tworelations}
Let us use Proposition~\ref{proposition:doubleMutation} to show that the algebra $\Lambda'$ given by
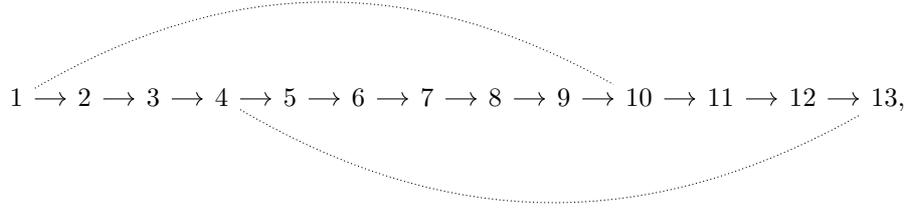
\begin{equation*}
\begin{tikzcd}[column sep = small]
1 \ar[r] \ar[rrrrrrrrr, bend left, densely dotted, -] & 2 \ar[r] & 3 \ar[r] & 4 \ar[r] \ar[rrrrrrrrr, bend right, densely dotted, -]& 5 \ar[r] & 6 \ar[r] & 7 \ar[r] & 8 \ar[r] & 9 \ar[r] & 10 \ar[r] & 11 \ar[r] & 12 \ar[r] & 13,
\end{tikzcd}
\end{equation*}
which appeared in the Introduction, is derived equivalent to the algebra $\Lambda$ given by
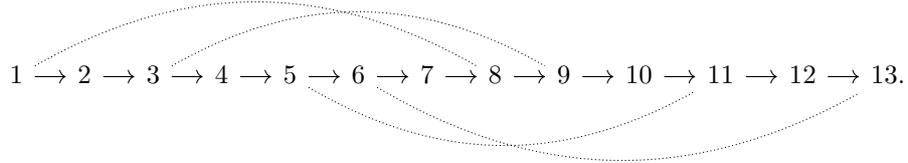
\begin{equation*}
\begin{tikzcd}[column sep = small]
1 \ar[r] \ar[rrrrrrr, bend left, densely dotted, -]& 2 \ar[r] & 3 \ar[r] \ar[rrrrrr, bend left, densely dotted, -] & 4 \ar[r] & 5 \ar[r] \ar[rrrrrr, bend right, densely dotted, -]& 6 \ar[r] \ar[rrrrrrr, bend right, densely dotted, -] & 7 \ar[r] & 8 \ar[r] & 9 \ar[r] & 10 \ar[r] & 11 \ar[r] & 12 \ar[r] & 13.
\end{tikzcd}
\end{equation*}

Starting with the latter we apply $R_1(-)$ twice and then $L_{13}(-)$ twice, which yields the following chain of derived equivalent Nakayama algebras.

\medskip
$R_1(\Lambda)\colon$
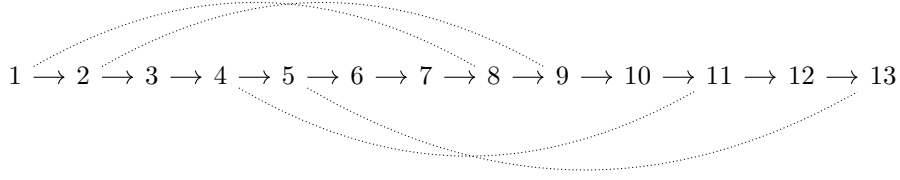
\begin{equation*}
\begin{tikzcd}[column sep = small]
1 \ar[r] \ar[rrrrrrr, bend left, densely dotted, -] & 2 \ar[r] \ar[rrrrrrr, bend left, densely dotted, -] & 3 \ar[r] & 4 \ar[r] \ar[rrrrrrr, bend right, densely dotted, -]& 5 \ar[r] \ar[rrrrrrrr, bend right, densely dotted, -]& 6 \ar[r] & 7 \ar[r] & 8 \ar[r] & 9 \ar[r] & 10 \ar[r] & 11 \ar[r] & 12 \ar[r] & 13
\end{tikzcd}
\end{equation*}

$R_1^2(\Lambda)\colon$
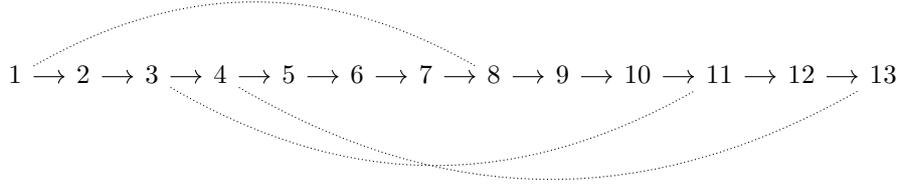
\begin{equation*}
\begin{tikzcd}[column sep = small]
1 \ar[r] \ar[rrrrrrr, bend left, densely dotted, -] & 2 \ar[r] & 3 \ar[r] \ar[rrrrrrrr, bend right, densely dotted, -]& 4 \ar[r] \ar[rrrrrrrrr, bend right, densely dotted, -]& 5 \ar[r] & 6 \ar[r] & 7 \ar[r] & 8 \ar[r] & 9 \ar[r] & 10 \ar[r] & 11 \ar[r] & 12 \ar[r] & 13
\end{tikzcd}
\end{equation*}

$L_{13}(R_1^2(\Lambda))\colon$
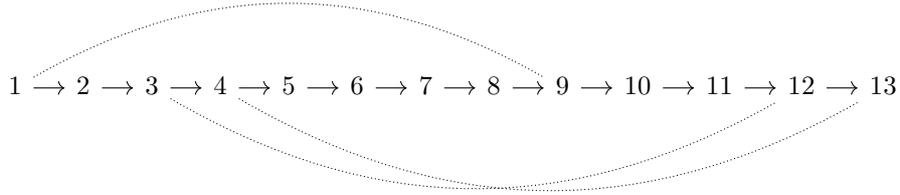
\begin{equation*}
\begin{tikzcd}[column sep = small]
1 \ar[r] \ar[rrrrrrrr, bend left, densely dotted, -] & 2 \ar[r] & 3 \ar[r] \ar[rrrrrrrrr, bend right, densely dotted, -]& 4 \ar[r] \ar[rrrrrrrrr, bend right, densely dotted, -]& 5 \ar[r] & 6 \ar[r] & 7 \ar[r] & 8 \ar[r] & 9 \ar[r] & 10 \ar[r] & 11 \ar[r] & 12 \ar[r] & 13
\end{tikzcd}
\end{equation*}

$L_{13}^2(R_1^2(\Lambda))=\Lambda'\colon$
\begin{equation*}
\begin{tikzcd}[column sep = small]
1 \ar[r] \ar[rrrrrrrrr, bend left, densely dotted, -] & 2 \ar[r] & 3 \ar[r] & 4 \ar[r] \ar[rrrrrrrrr, bend right, densely dotted, -]& 5 \ar[r] & 6 \ar[r] & 7 \ar[r] & 8 \ar[r] & 9 \ar[r] & 10 \ar[r] & 11 \ar[r] & 12 \ar[r] & 13
\end{tikzcd}
\end{equation*}
Example~\ref{example:A13notPWH} will show that $\Lambda$, hence each of these algebras, is non-piecewise hereditary.
\end{example}

\section{Obstructions to being piecewise hereditary}

A \textbf{path} in a triangulated category is a tuple $(X_0, \ldots, X_n)$ of indecomposable objects such that $\Hom(X_{i-1},X_i)\neq0$ for each $i\in\{1,\ldots n\}$.

In the sequel we will denote such a path by $X_0 \leadsto X_n$.

\begin{remark}
What we call a path here is usually called a \textbf{strong path}. e.g.\ in \cite{MR3896230,MR2413349}.
\end{remark}

The basic observation is easy enough:
\begin{lemma}\label{lemma:pathimpliesnotpwh}
Let $\Lambda$ be an algebra. If $\Db(\Lambda)$ contains an indecomposable object $X$ with a path $X[i]\leadsto X$ for some $i\ge1$, then $\Lambda$ is not piecewise hereditary.
\end{lemma}
\begin{proof}
If $\mathcal H$ is a hereditary abelian category and $A,B\in \Db(\mathcal H)$ are indecomposable, then there are integers $a,b$ such that $A\in\mathcal H[a]$ and $B\in\mathcal H[b]$. Moreover, the existence of a non-zero morphism $A\to B$ implies $b\in\{a, a+1\}$.

In particular, for any indecomposable object $X$ in $\Db(\mathcal H)$ there can be no path $X[1]\leadsto X$, and the claim of the lemma follows.
\end{proof}

\subsection*{Coarse versus fine}

In some cases, a simple-minded approach lets us produce paths of the form $X[i]\leadsto X$ in the derived category:

\begin{example}\label{example:simpleminded}
Consider the algebra $\Lambda$ given by the following quiver with relations.
\[ \begin{tikzcd}
1 \ar[r] \ar[rrr,bend left,densely dotted,-] & 2 \ar[r] \ar[rrr,bend right,densely dotted,-]& 3 \ar[r] & 4 \ar[r] \ar[rrr,bend left,densely dotted,-] & 5 \ar[r] \ar[rrr,bend right,densely dotted,-] & 6 \ar[r] & 7 \ar[r]  \ar[rrr,bend left,densely dotted,-] & 8 \ar[r] & 9 \ar[r] & 10
\end{tikzcd} \]
Starting with $P_1$ we can produce a ``coarsest" indecomposable complex of projectives, in the sense that we make the difference between neighboring indices maximal: \[P_9\to P_6 \to P_3 \to P_1\] At the other extreme there is a ``finest" indecomposable complex, namely \[P_{10}\to P_8 \to P_7 \to P_5 \to P_4 \to P_2 \to P_1.\]
For each non-zero chain map of the form
\[\begin{tikzcd}
P_9 \ar[r] & P_6 \ar[r] & P_3 \ar[r] & P_1 & \\
P_{10} \ar[u,"\neq0"] \ar[r] & P_8 \ar[r] & P_7 \ar[r] & P_5 \ar[r] & P_4 \ar[r] & P_2 \ar[r] & P_1
\end{tikzcd}\]
the mapping cone is a complex of the form
\[
P_{10} \xto{\left(\begin{smallmatrix}\ast\\\ast\end{smallmatrix}\right)} P_9\oplus P_8 \xto{\left(\begin{smallmatrix}\ast&0\\0&\ast\end{smallmatrix}\right)} P_6\oplus P_7 \xto{\left(\begin{smallmatrix}\ast&0\\0&\ast\end{smallmatrix}\right)} P_3\oplus P_5 \xto{\left(\begin{smallmatrix}\ast&0\\0&\ast\end{smallmatrix}\right)} P_1\oplus P_4 \xto{\left(\begin{smallmatrix}0&\ast\end{smallmatrix}\right)} P_2 \stackrel{\ast}\to P_1,
\]
where each $\ast$ indicates a non-zero morphism of modules. After removing the rightmost term \( P_1 \) we obtain a complex which
\begin{enumerate}
\item admits $P_1$ as a subcomplex in degree $i$ and as a quotient in degree $i+1$; and
\item is indecomposable.
\end{enumerate}
In particular there is a path $P_1[1]\leadsto P_1$ in $\Db(\Lambda)$ as indicated by the diagram
\[\begin{tikzcd}[ampersand replacement=\&]
\&\&\&\&P_1\ar[d,"{\left(\begin{smallmatrix}1\\0\end{smallmatrix}\right)}"]\&\\
P_{10} \ar[r] \& P_9\oplus P_8 \ar[r] \& P_6\oplus P_7 \ar[r] \& P_3\oplus P_5 \ar[r] \& P_1\oplus P_4 \ar[r] \& P_2 \ar[d,"\ast"]\\
\&\&\&\&\&P_1,
\end{tikzcd}\]
so by Lemma~\ref{lemma:pathimpliesnotpwh} the algebra $\Lambda$ is not piecewise hereditary.
\end{example}

In the sequel we will represent the situation in Example~\ref{example:simpleminded} by a diagram
\begin{equation}\label{diagram:representspath}
\begin{tikzcd}
 & P_9 \ar[r] & P_6 \ar[r] & P_3 \ar[r] & P_1 & \\
P_{10} \ar[ur,dashed] \ar[r] & P_8 \ar[r] & P_7 \ar[r] & P_5 \ar[r] & P_4 \ar[r] & P_2.
\end{tikzcd} \tag{$\dagger$}
\end{equation}

\begin{remark}
The naive approach from Example~\ref{example:simpleminded} does not always work:
\begin{enumerate}
\item The ``coarse" complex need not be sufficiently short compared to the ``fine" complex even if the algebra is not piecewise hereditary; and, perhaps more subtly,
\item the mapping cone is indecomposable if and only if the indices in the associated diagram (\ref{diagram:representspath}) are sufficiently intertwined. Explicitly---and informally---Example~\ref{example:simpleminded} works because the indices go from being larger in the top row ($9>8$) to being larger in the bottom row ($6<7$). We can think of this twist as somehow making the endomorphism ring of the complex small enough so that no non-zero idempotent can appear.
\end{enumerate}
\end{remark}

Let us illustrate the second point with the following non-example, before making the positive statement clear.

\begin{nonexample}
Consider $k\mathbb A_{11}/(\rad)^4$. This algebra is known to not be piecewise hereditary by \cite{MR2736030}, a fact which we will also see here in Example~\ref{example:A11/rad4}.

The coarsest and finest complexes in this case are
\[ P_{10} \to P_7 \to P_4 \to P_1 \text{ and } P_{10} \to P_9 \to P_6 \to P_5 \to P_2 \to P_1, \]
respectively. This might lead us to hope that we get a path
\[\begin{tikzcd}[ampersand replacement=\&,column sep=4em]
\&\&\&P_1\ar[d]\&\\
P_{10} \ar[r,"{\left(\begin{smallmatrix} \ast \\ \ast \end{smallmatrix}\right)}"] \& P_7\oplus P_9 \ar[r,"{\left(\begin{smallmatrix} \ast &0\\0& \ast \end{smallmatrix}\right)}"] \& P_4\oplus P_6 \ar[r,"{\left(\begin{smallmatrix} \ast &0\\0& \ast \end{smallmatrix}\right)}"] \& P_1\oplus P_5 \ar[r,"{\left(\begin{smallmatrix}0& \ast \end{smallmatrix}\right)}"] \& P_2 \ar[d]\\
\&\&\&\&P_1
\end{tikzcd}\]
in the derived category. However, here the complex in the middle is \emph{not} indecomposable.
\end{nonexample}

\begin{lemma} \label{lemma:is indec cx}
Let
\begin{align*}
 P_{c_1} \to P_{c_2} \to P_{c_3} \to \cdots \to P_{c_{\ell_c}} & \,\, and \\
 P_{f_1} \to P_{f_2} \to P_{f_3} \to \cdots \to P_{f_{\ell_f}} &
\end{align*}
be indecomposable complexes of indecomposable projectives over a Nakayama algebra, such that \( c_1 = f_1 \). Assume that
\begin{itemize}
\item \( c_2 < f_2 \) and that
\item there is some $i$ such that $ c_i > f_i $ and $c_j = f_j$ for each $ 2 < j < i$.
\end{itemize}
Then the complex
\[ P_{c_1} \xto{\left(\begin{smallmatrix}\ast\\\ast\end{smallmatrix}\right)} P_{c_2} \oplus P_{f_2}  \xto{\left(\begin{smallmatrix}\ast&0\\0&\ast\end{smallmatrix}\right)} P_{c_3} \oplus P_{f_3} \to \cdots \]
is indecomposable.
\end{lemma}

\begin{proof}

Let us first consider the case that \( c_3 > f_3 \).
We calculate the endomorphism ring of the complex. Note that indecomposable projective modules have endomorphism ring \( k \), and \( \Hom(P_a, P_b) \) vanishes unless \( a \geq b \). So by assumption an endomorphism of our complex is of the form
\[\begin{tikzcd}[ampersand replacement=\&,column sep=4em,row sep=3em]
P_{c_1} \ar[r] \ar[d,"\lambda_1"] \& P_{c_2}\oplus P_{f_2} \ar[r] \ar[d,"{\left(\begin{smallmatrix} \lambda_2 &0\\\ast& \lambda'_2 \end{smallmatrix}\right)}"] \& P_{c_3}\oplus P_{f_3} \ar[r] \ar[d,,"{\left(\begin{smallmatrix} \lambda_3 &\ast\\0& \lambda'_3 \end{smallmatrix}\right)}"]\& \cdots \\
P_{c_1} \ar[r] \& P_{c_2}\oplus P_{f_2} \ar[r] \& P_{c_3}\oplus P_{f_3} \ar[r] \& \cdots.
\end{tikzcd}\]
Commutativity of the second square forces both entries marked $\ast$ to vanish. Together with commutativity of the square on the left this is easily seen to imply that \[\lambda_1=\lambda_2=\lambda'_2=\lambda_3=\lambda'_3,\] and this common scalar will appear on each diagonal entry of each subsequent matrix in the endomorphism. If the scalar is zero then the endomorphism is nilpotent. This shows that the endomorphism ring is local, so the complex is indecomposable.

The situation with one or more equal terms in the two sequences only introduces additional matrices to the argument above that need to be equal to both their left and right neighbors. Thus the same argument shows that all diagonal entries of all matrices are equal.
\end{proof}

Incorporating this last lemma into the strategy of Example~\ref{example:simpleminded} naturally leads us to considering the following two sequences of integers.

\begin{construction} \label{const:sequences}
Let \( \Lambda \) be a Nakayama algebra. We will iteratively construct a coarse and a fine sequence, $c=(c_i)$ and $f=(f_i)$ respectively, running along the quiver of \( \Lambda \), from right to left.

Let \( c_1 = f_1 \) be the end-point of the rightmost relation of length at least \( 3 \), excluding a possible relation \( (c_1 - 3) \dashrightarrow c_1 \) where there is another relation ending in \( c_1 - 1 \), but no relation ending in \( c_1 - 2 \).

Let \( c_2 \) be the largest number between the starting-point of this relation \( + 2 \) and \( c_1 \), such that no relation ends in \( c_2 \). If no such number exists we set \( c_2 = c_1 - 1 \). Let \( f_2 = c_2 - 1 \).

Now we set
\begin{align*}
c_{i+1} & = (\text{start-point of the last relation ending before or at } c_i) + 1 \text{, and} \\
f_{i+1} & = \minimum \big\{ (\text{start-point of the last relation ending before or at } f_{i-1}), \\
& \qquad\qquad\! (\text{end-point of the last relation ending before or at } f_i) - 1 \big\}
\end{align*}
while the required relations exist. We denote by \( \ell_c \) and \( \ell_f \) the lengths of these two sequences, respectively.
\end{construction}

\begin{remark}
The sequences are constructed in such a way that
\[ P_{c_1} \to P_{c_2} \to \cdots \to P_{c_{\ell}} \to P_1 \]
and the similar diagram with \( P_{f_i} \) are indecomposable complexes.

The technical condition on \( c_2 \) ensures that the condition of Lemma~\ref{lemma:is indec cx} is met, at least provided \( c_i > f_i \) for sufficiently large \( i \).
\end{remark}

Thus, we have all the ingredients for the following result.

\begin{proposition} \label{prop.coarse_fine}
In the situation of Construction~\ref{const:sequences}, if \( \ell_f \geq \ell_c + 2 \), then \( \Lambda \) is not piecewise hereditary.
\end{proposition}

\begin{proof}
We invoke Lemma~\ref{lemma:pathimpliesnotpwh}: It suffices to consider the path
\[\begin{tikzcd}[ampersand replacement=\&,column sep=2.5em]
\&\&\&\&P_1\ar[d]\&\\
P_{c_1} \ar[r,"{\left(\begin{smallmatrix} \ast \\ \ast \end{smallmatrix}\right)}"] \& P_{f_2} \oplus P_{c_2} \ar[r,"{\left(\begin{smallmatrix} \ast &0\\0& \ast \end{smallmatrix}\right)}"] \& P_{f_3} \oplus P_{c_3} \ar[r,"{\left(\begin{smallmatrix} \ast &0\\0& \ast \end{smallmatrix}\right)}"] \& \cdots \ar[r,"{\left(\begin{smallmatrix} \ast &0\\0& \ast \end{smallmatrix}\right)}"]  \& P_1\oplus P_{f_{\ell_c+1}} \ar[r,"{\left(\begin{smallmatrix}0& \ast \end{smallmatrix}\right)}"] \& P_{f_{\ell_c+2}} \ar[d]\\
\&\&\&\&\& \vdots \ar[d] \\
\&\&\&\&\& P_1
\end{tikzcd}\]
in $\Db(\Lambda)$ and note that the second line is indecomposable by the discussion above.
\end{proof}

\begin{examplesimplemindedrevisited}
In this case
\[ c = (10, 9, 6, 3) \text{ and } f = (10, 8, 7, 5, 4, 2). \]
\end{examplesimplemindedrevisited}

\begin{example}
Consider \( k \mathbb{A}_n / (\rad)^m \) with \( m \geq 3 \). In the parlance of Construction~\ref{const:sequences} we then have
\begin{align*}
c & = (n, n-1, n-m, n-2m+1, n-3m+2, \, \ldots) \text{ and} \\
f & = (n, n-2, n-m, n-2-m, n-2m, n-2-2m, \, \ldots),
\end{align*}
both ending when the entries become \( \leq 1 \). A straightforward calculation shows that \[\text{ $\ell_c = \left\lfloor\frac{n-3}{m-1} \right\rfloor + 2$ and $\ell_f = \maximum \left\{ 2 \left\lfloor \frac{n-4}{m} \right\rfloor + 2, \,2 \left\lfloor \frac{n-2}{m} \right\rfloor + 1 \right\}$}.\]
It follows that \( k \mathbb{A}_n / (\rad)^3 \) is not piecewise hereditary for \( n \geq 14 \), while for \( m \geq 4 \) the algebra \( k \mathbb{A}_n / (\rad)^m \) is not piecewise hereditary for \( n \geq 2m+2 \).

While these bounds are not at all tight, the argument does give an easy proof that for given \( m \) these algebras will eventually stop being piecewise hereditary.
\end{example}

\subsection*{Auslander--Reiten translation}
The bounded derived category of a Nakayama algebra has almost split triangles; $\tau X$ denotes the Auslander--Reiten translate of $X$.

\begin{lemma}\label{lemma:taupathimpliesnotpwh}
Suppose there exists an indecomposable object $X$ in $\Db(\Lambda)$ with \[\text{$\tau^n X = X[m]$ for some $n,m \ge 1$.}\] Then $\Lambda$ is not piecewise hereditary.
\end{lemma}
\begin{proof}
Consider the almost split triangle \[\tau X \stackrel{f}\to Z \stackrel{g}\to X \to \tau X[1], \] and write $Z=Z_1 \oplus \dots \oplus Z_k$ with each $Z_i$ indecomposable. Clearly, by assumption and Lemma~\ref{lemma:pathimpliesnotpwh} it suffices to show that there is a common $i\in\{1,\ldots,k\}$ such that the canonical components $f_i\colon \tau X \to Z_i$ and $g_i\colon Z_i \to X$ are both non-zero.

In fact, we can show that $g_i$ and $f_i$ are both non-zero for each $i$:

Let us show that each $f_i$ is non-zero. Suppose some $f_i=0$. Without loss of generality we may assume $i=1$. Now, writing $Z_{\ast}=Z_2 \oplus \dots \oplus Z_n$ we have the solid part of the following commutative diagram.
\[\begin{tikzcd}[ampersand replacement=\&]
0 \ar[d] \ar[r] \& Z_1 \ar[d,"{\left(\begin{smallmatrix}1\\0\end{smallmatrix}\right)}"] \ar[r,"1"] \& Z_1 \ar[r] \ar[d,dashed,"\zeta"] \& 0 \ar[d]\\
\tau X \ar[d] \ar[r,"f"] \& Z_1\oplus Z_{\ast} \ar[d,"{\left(\begin{smallmatrix}1&0\end{smallmatrix}\right)}"] \ar[r,"g"] \& X \ar[r] \ar[d,dashed,"\zeta'"]\& \tau X[1] \ar[d]\\
0 \ar[r] \& Z_1 \ar[r,"1"] \& Z_1 \ar[r] \& 0
\end{tikzcd}\]
This may be completed to morphisms of triangles as indicated by the dashed arrows, and by the Five Lemma the composition $\zeta'\circ\zeta$ is invertible. In particular this means that $\zeta=g_1$ is split mono. But then, since $Z_1$ and $X$ are indecomposable, $g_1$ must be invertible. It follows that $g$ is split epi and we have a contradiction.

A dual argument shows that each $g_i$ is non-zero.
\end{proof}

\begin{example}\label{example:A11/rad5}
Consider $k \mathbb A_{11}/(\rad)^5$. Successive application of $\tau$ to $P_1$ in the derived category yields:
\[ \begin{tikzcd}[row sep=0]
\textcolor{gray}{\text{degree}} & & \textcolor{gray}{-4} & \textcolor{gray}{-3} & \textcolor{gray}{-2} & \textcolor{gray}{-1} & \textcolor{gray}{0} & \textcolor{gray}{1} & \\
P_1\colon &  &  &  &  & & P_1 & & \\
\tau(P_1)\colon &  &  & P_{11} \ar[r] & P_7 \ar[r] & P_6 \ar[r] & P_2 \ar[r] & P_1 & \\
\tau^2(P_1)\colon &  &  &  & P_8 &  & &  & \\
\tau^3(P_1)\colon &  &  &  &  & P_4 & &  & \\
\tau^4(P_1)\colon &  & P_{11} \ar[r] & P_{10} \ar[r] & P_6 \ar[r] & P_5 \ar[r] & P_1 &  & \\
\tau^5(P_1)\colon &  &  & P_{11} &  &  & &  & \\
\tau^6(P_1)\colon &  &  &  & P_7 & & &  & \\
\tau^7(P_1)\colon &  &  &  &  & P_3 & &  & \\
\tau^8(P_1)\colon &  & P_{11} \ar[r] & P_9 \ar[r] & P_6 \ar[r] & P_4 \ar[r] & P_1 &  & \\
\tau^9(P_1)\colon &  &  & P_{10} &  &  & &  & \\
\tau^{10}(P_1)\colon &  &  &  & P_6 & & &  & \\
\tau^{11}(P_1)\colon &  &  &  &  & P_2 & &  & \\
\tau^{12}(P_1)\colon &  & P_{11} \ar[r] & P_8 \ar[r] & P_6 \ar[r] & P_3 \ar[r] & P_1 &  & \\
\tau^{13}(P_1)\colon &  &  & P_9 &  &  & &  & \\
\tau^{14}(P_1)\colon &  &  &  & P_5 & & &  & \\
\tau^{15}(P_1)\colon &  &  &  &  & P_1 & &  &
\end{tikzcd} \]
Since $\tau^{15}(P_1)= P_1[1]$, Lemma~\ref{lemma:taupathimpliesnotpwh} reveals that this algebra is not piecewise hereditary.
\end{example}

\begin{example}\label{example:A11/rad4}
(Re)consider the algebra $k\mathbb A_{11}/(\rad)^4$. Also here $\tau^{15}(P_1)=P_1[1]$ in the derived category, so this algebra is not piecewise hereditary by Lemma~\ref{lemma:taupathimpliesnotpwh}.
\end{example}

In some cases when $\tau^mX=X[n]$ for some indecomposable $X$, it is fairly easy to picture a path $X[\ge1]\leadsto X$:

\begin{exampleA11/rad5revisited} The table describing the $\tau^i(P_1)$ in the derived category of $k\mathbb A_{11}/(\rad)^5$ suggests that we should consider the following (\ref{diagram:representspath})-like diagram.
\begin{equation}\label{equation:nicepath}
\begin{tikzcd}
&  & P_9 \ar[r] & P_5 \ar[r] & P_1 &  & \\
& P_{11} \ar[ur,dashed] \ar[r] & P_8 \ar[r] & P_6 \ar[r] & P_3 \ar[r] & P_1 & \\
&  & P_{10} \ar[r] & P_6 \ar[r] & P_2 \ar[ur,dashed] &  & \\
& P_{11} \ar[ur,dashed] \ar[r] & P_9 \ar[r] & P_6 \ar[r] & P_4 \ar[r] & P_1 & \\
&  & & P_7 \ar[r] & P_3 \ar[ur,dashed] &  & \\
&  & P_{10} \ar[ur,dashed] \ar[r] & P_6 \ar[r] & P_5 \ar[r] & P_1 & \\
&  & & P_8 \ar[r] & P_4 \ar[ur,dashed] &  & \\
& & P_{11} \ar[ur,dashed] \ar[r] & P_7 \ar[r] & P_6 \ar[r] & P_2 
\end{tikzcd} \tag{$\dagger\dagger$}
\end{equation}
Notice that
\begin{itemize}
\item $P_1$ appears as a subcomplex of the top row in degree $i$ and as a quotient of the complex in the bottom row in degree $i+1$, and that
\item for each pair of neighboring rows, the indices are sufficiently intertwined so that the corresponding mapping cone will be an indecomposable complex.
\end{itemize}
In particular the above represents a path $P_1[1]\leadsto P_1$ in $\Db(k\mathbb A_{11}/(\rad)^5)$.
\end{exampleA11/rad5revisited}

On the basis of (\ref{equation:nicepath}) one might start hoping for general existence criteria for paths of the form $X[\ge1]\leadsto X$ i.e.,\ not involving $\tau$: Maybe there is a more sophisticated version of the ``coarse versus fine" construction?

Intriguingly, however, the transparency exhibited in the case of $k\mathbb A_{11}/(\rad)^5$ seems to be rather exceptional. In particular, the process of slow-but-steady consolidation which worked like magic in Example~\ref{example:A11/rad5} does not lend itself to immediate generalization---at least not in the most naive of ways---and so any real conceptual understanding of the whole situation escapes us still.

\begin{example}\label{example:A11tworelations} Let $\Lambda$ be the algebra given by
\[ \begin{tikzcd}[column sep=2em]
1 \ar[r] & 2 \ar[r] \ar[rrrrrr,bend left,densely dotted,-] & 3 \ar[r] & 4 \ar[r] \ar[rrrrrr,bend right,densely dotted,-] & 5 \ar[r] & 6 \ar[r] & 7 \ar[r] & 8 \ar[r] & 9 \ar[r] & 10 \ar[r] & 11.
\end{tikzcd} \]
Then $\tau^4(P_6)=P_6[1]$ in $\Db(\Lambda)$, so $\Lambda$ is not piecewise hereditary by Lemma~\ref{lemma:taupathimpliesnotpwh}. But here the situation is somehow less lucid (compare to Example~\ref{example:A11/rad5}):
\[ \begin{tikzcd}[row sep=0]
\textcolor{gray}{\text{degree}} & & \textcolor{gray}{-2} & \textcolor{gray}{-1} & \textcolor{gray}{0} & \textcolor{gray}{1} \\
P_6\colon &  &  &  &  P_6 & & \\
\tau(P_6)\colon &    & & P_8 \ar[r] & P_7 \ar[r] & P_1 & \\
\tau^2(P_6)\colon &    & P_{10} \ar[r] & P_9 \ar[r] & P_3 \ar[r] & P_2 & \\
\tau^3(P_6)\colon &    & P_{11} \ar[r] & P_5 \ar[r] & P_4 & & \\
\tau^4(P_6)\colon &    & & P_6 & & & \\
\end{tikzcd} \]
From this table one can certainly produce a (\ref{equation:nicepath})-like picture, but we do contend that the associated path $P_6[1]\leadsto P_6$, although relatively short, is less conspicuous.
\end{example}

Perhaps not so surprisingly, the complexes appearing can be much easier or harder to interpret for different representatives of the same derived equivalence class. Let us illustrate this behavior.

\begin{example}\label{example:equivToA11_5}
In Example~\ref{example:mutationToA11_5} we showed that the algebra $\Lambda$ given by the following quiver with relations is derived equivalent to $k \mathbb A_{11}/(\rad)^5$.
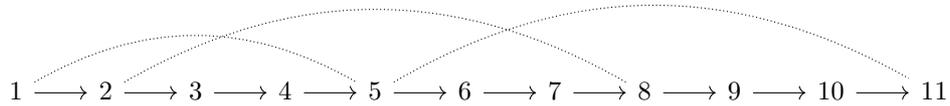
\[ \begin{tikzcd}[column sep=2em]
1 \ar[r] \ar[rrrr,bend left,densely dotted,-] & 2 \ar[r] \ar[rrrrrr,bend left,densely dotted,-] & 3 \ar[r] & 4 \ar[r] & 5 \ar[r] \ar[rrrrrr,bend left,densely dotted,-] & 6 \ar[r] & 7 \ar[r] & 8 \ar[r] & 9 \ar[r] & 10 \ar[r] & 11
\end{tikzcd} \]
Here we also have $\tau^{15}(P_1)=P_1[1]$ in the derived category of $\Lambda$, but the explicit calculation of Auslander--Reiten translations is much more painful than in Example~\ref{example:A11/rad5}), and writing down a nice indecomposable complex like in \eqref{equation:nicepath} does not seem straightforward.
Indeed, in $\Db(\Lambda)$ we have the following objects.

\[ \begin{tikzcd}[row sep=0,column sep=2em]
\textcolor{gray}{\text{degree}} &[-25pt] &[-10pt] \textcolor{gray}{-4} &[-10pt] \textcolor{gray}{-3} & \textcolor{gray}{-2} & \textcolor{gray}{-1} & \textcolor{gray}{0} &[-10pt] \textcolor{gray}{1} & \\
P_1\colon &  &  &  &  & & P_1 & & \\
\tau(P_1)\colon &  & & P_{11} \ar[r] & P_8 \ar[r] & P_5 \ar[r] & P_2 \ar[r] & P_1 & \\
\tau^2(P_1)\colon &  & & & P_9 &  &  & & \\
\tau^3(P_1)\colon &  & & P_{11} \ar[r] & P_{10} \ar[r] & P_3 & & & \\
\tau^4(P_1)\colon &  & & & P_5\oplus P_6 \ar[r] & P_3\oplus P_4 \ar[r] & P_1 && \\
\tau^5(P_1)\colon &  & & P_8\oplus P_{11} \ar[r] & P_6\oplus P_7 \ar[r] & P_2\oplus P_4 \ar[r] & P_1 && \\
\tau^6(P_1)\colon &  & & P_9 \ar[r] & P_6\oplus P_7 \ar[r] & P_2 \ar[r] & P_1 && \\
\tau^7(P_1)\colon &  & & P_{10} \ar[r] & P_7 \ar[r] & P_2 & && \\
\tau^8(P_1)\colon &  & & & P_5 \ar[r] & P_2 \ar[r] & P_1&& \\
\tau^9(P_1)\colon &  & & P_{11} \ar[r] & P_6 \ar[r] & P_3 & && \\
\tau^{10}(P_1)\colon &  & & P_8 \ar[r] & P_5\oplus P_6 \oplus P_7 \ar[r] & P_2 \oplus P_4 \ar[r] & P_1&& \\
\tau^{11}(P_1)\colon &  & & P_8\oplus P_9 \ar[r] & P_6 \oplus P_7 \ar[r] & P_2 \oplus P_2 \ar[r] & P_1&& \\
\tau^{12}(P_1)\colon &  & P_{11} \ar[r] & P_9\oplus P_{10} \ar[r] & P_5 \oplus P_7 \ar[r] & P_2 \ar[r] & P_1&& \\
\tau^{13}(P_1)\colon &  &  & P_{10} \ar[r] & P_3 \ar[r] & P_2 & && \\
\tau^{14}(P_1)\colon &  &  &  & P_4 &  & && \\
\tau^{15}(P_1)\colon &  &  &  & & P_1 & &&
\end{tikzcd} \]
\end{example}

Now we employ Lemma~\ref{lemma:taupathimpliesnotpwh} to gather more examples of Nakayama algebras which are not piecewise hereditary, the first of which will serve as a base case in the sequel.

\begin{example}\label{example:A9} Let $\Lambda$ be given by
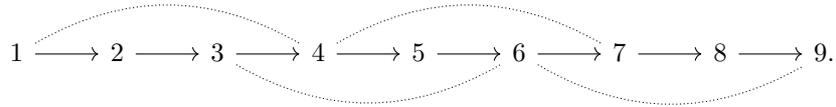
\[ \begin{tikzcd}
1 \ar[r] \ar[rrr,bend left,densely dotted,-] & 2 \ar[r] & 3 \ar[r] \ar[rrr,bend right,densely dotted,-] & 4 \ar[r] \ar[rrr,bend left,densely dotted,-] & 5 \ar[r] & 6 \ar[r] \ar[rrr,bend right,densely dotted,-] & 7 \ar[r] & 8 \ar[r] & 9.
\end{tikzcd} \]
Then $\tau^4(P_2)=P_2[1]$ in $\Db(\Lambda)$, so $\Lambda$ is not piecewise hereditary.
\end{example}

\begin{example}\label{example:A10single} Let $\Lambda$ be given by
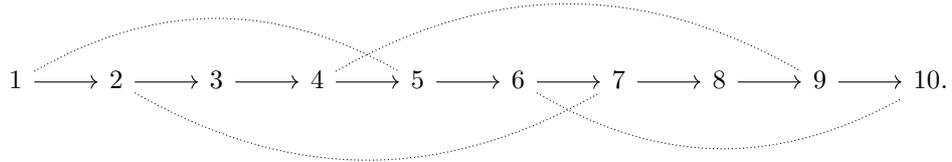
\[ \begin{tikzcd}
1 \ar[r] \ar[rrrr,bend left,densely dotted,-] & 2 \ar[r] \ar[rrrrr,bend right,densely dotted,-] & 3 \ar[r] & 4 \ar[r] \ar[rrrrr,bend left,densely dotted,-] & 5 \ar[r] & 6 \ar[r] \ar[rrrr,bend right,densely dotted,-] & 7 \ar[r] & 8 \ar[r] & 9 \ar[r] & 10.
\end{tikzcd} \]
Then $\tau^7(P_3)=P_3[2]$ in $\Db(\Lambda)$, so $\Lambda$ is not piecewise hereditary.
\end{example}

\begin{example}\label{example:A10double} Let $\Lambda$ be
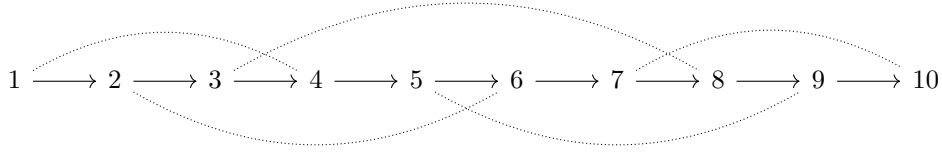
\[ \begin{tikzcd}
1 \ar[r] \ar[rrr,bend left,densely dotted,-] & 2 \ar[r] \ar[rrrr,bend right,densely dotted,-] & 3 \ar[r] \ar[rrrrr,bend left,densely dotted,-] & 4 \ar[r] & 5 \ar[r] \ar[rrrr,bend right,densely dotted,-] & 6 \ar[r] & 7 \ar[r] \ar[rrr,bend left,densely dotted,-] & 8 \ar[r] & 9 \ar[r] & 10
\end{tikzcd} \]
and let $\Lambda'$ be given by
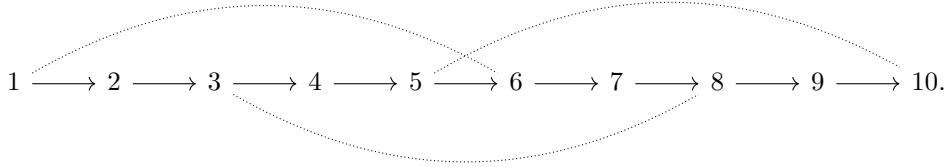
\[ \begin{tikzcd}
1 \ar[r] \ar[rrrrr,bend left,densely dotted,-] & 2 \ar[r] & 3 \ar[r] \ar[rrrrr,bend right,densely dotted,-] & 4 \ar[r] & 5 \ar[r] \ar[rrrrr,bend left,densely dotted,-] & 6 \ar[r] & 7 \ar[r] & 8 \ar[r] & 9 \ar[r] & 10.
\end{tikzcd} \]
Then $\tau^9(P_2)=P_2[1]$ holds true both in $\Db(\Lambda)$ and in $\Db(\Lambda')$, so $\Lambda$ and $\Lambda'$ are both non-piecewise hereditary.
\end{example}

\begin{remark}\label{remark:Coxeter}
The non-piecewise hereditary algebras $\Lambda$ and $\Lambda'$ from Example~\ref{example:A10double} both have the same Coxeter polynomial as the piecewise hereditary $k\mathbb A_{10}/(\rad)^9$, namely $T^{10}+T^9+T+1$.
\end{remark}

Finally, we gather the cases on which Proposition~\ref{proposition:HS} will be built:

\begin{liste}\label{liste:baseforHS}
The following five algebras are not piecewise hereditary.
\begin{enumerate}
\item\label{HSlist:A11/4} $k\mathbb A_{11}/(\rad)^4$
\item\label{HSlist:A11/5} $k\mathbb A_{11}/(\rad)^5$
\item\label{HSlist:A12/3} $k\mathbb A_{12}/(\rad)^3$
\item\label{HSlist:A12/6} $k\mathbb A_{12}/(\rad)^6$
\item\label{HSlist:A12/7} $k\mathbb A_{12}/(\rad)^7$
\end{enumerate}
Note that we have already seen (\ref{HSlist:A11/4}) and (\ref{HSlist:A11/5}) in Example~\ref{example:A11/rad4} and Example~\ref{example:A11/rad5}, respectively. To argue why the algebras (\ref{HSlist:A12/3}), (\ref{HSlist:A12/6}) and (\ref{HSlist:A12/7}) are not piecewise hereditary, we assert that $\tau^{21}(P_1)=P_1[1]$ holds true in each of the respective derived categories.
\end{liste}

\section{Extensions}

The following starting point of our discussion here is well-known.

\begin{proposition}
Let \( \Lambda \) be a finite dimensional algebra with \( e \in \Lambda \) an idempotent. If \( \Lambda \) is piecewise hereditary, then so is \( e \Lambda e \).
\end{proposition}

\begin{proof}
First note that \( \Lambda \) being piecewise hereditary implies that \( \Lambda \), and thus also \( e \Lambda e \), is triangular. In particular both algebras have finite global dimension.

Now the functor \( - \otimes_{e \Lambda e}^{\mathbb{L}} e \Lambda \colon \Db(e \Lambda e) \to \Db(\Lambda)\) is fully faithful, and the claim follows immediately from Chen--Ringel's characterization of hereditary triangulated categories \cite{MR3896230}.
\end{proof}

Let us apply this result to Nakayama algebras.

\begin{corollary}\label{corollary:removevertex}
Assume \( \Lambda \) is a Nakayama algebra which is piecewise hereditary, and let \( i \) be a vertex in the quiver of \( \Lambda \). Let \( \Gamma \) be the Nakayama algebra obtained from \( \Lambda \) by removing the vertex \( i \), adding a new composite arrow if there is both an arrow into and out of \( i \), as well as extended relations for relations ending or starting in \( i \).

Then \( \Gamma \) is also piecewise hereditary.
\end{corollary}

Let us depict the construction of the corollary:

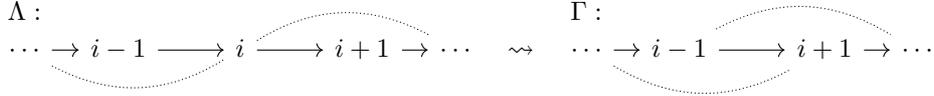
\[ \begin{tikzcd}[column sep=1.5em, row sep=0]
\Lambda: &[-5pt]&[+10pt]&[+10pt]&[-5pt]&[-10pt] &[-10pt]\Gamma:&[-5pt]&[10pt]&[-5pt]\\
\cdots \ar[r] \ar[rr,bend right,densely dotted,-] & i-1 \ar[r] & i \ar[r] \ar[rr,bend left,densely dotted,-] & i+1 \ar[r] & \cdots & \leadsto & \cdots \ar[r] \ar[rr,bend right,densely dotted,-] &  i-1 \ar[r] \ar[rr,bend left,densely dotted,-] & i+1 \ar[r] & \cdots
\end{tikzcd}\]

The corollary becomes a lot more flexible when being read ``backwards'':

\begin{corollary}\label{corollary:introducevertex}
Assume \( \Lambda \) is a Nakayama algebra which is not piecewise hereditary. Let \( \Gamma \) be a Nakayama algebra obtained from \( \Lambda \) by introducing an extra vertex \( \star \) between \( i \) and \( i+1 \), with the following possibilities for relations.
\begin{itemize}
\item If there is a relation in \( \Lambda \) ending in \( i+1 \), then we can either introduce a corresponding relation ending in \( \star \) in \( \Gamma \), or keep the relation as is. In the latter case we may introduce a relation starting anywhere before the start of the given relation, ending in \( \star \).
\item If there is a relation in \( \Lambda \) starting in \( i \), then we can either introduce a corresponding relation starting in \( \star \) in \( \Gamma \), or keep the relation as is. In the latter case we may introduce a relation starting in \( \star \) and ending anywhere after the end of the given relation.
\end{itemize}

Then $\Gamma$ is also not piecewise hereditary.
\end{corollary}

Before looking at examples, let us briefly record the simplified version of this result for the case that we add a vertex at the start or end of the quiver.

\begin{corollary} \label{corollary:introducevertexatend}
Suppose \( \Lambda \) is a Nakayama algebra which is not piecewise hereditary. Let \( \Gamma \) be a Nakayama algebra obtained from \( \Lambda \) by introducing an extra vertex \( \star \) either before \( 1 \) or after \( n \), together with the relations from \( \Lambda \) and any choice of a relation or no relation involving \( \star \).

Then \( \Gamma \) is also not piecewise hereditary.
\end{corollary}

Now let us look at some examples and applications of Corollary~\ref{corollary:introducevertex}.

\begin{example}
We saw in Example~\ref{example:A9} that the algebra given by the following quiver with relations is not piecewise hereditary.
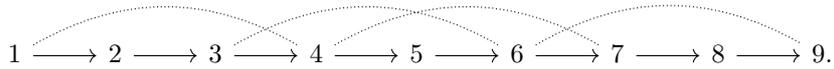
\[ \begin{tikzcd}
1 \ar[r] \ar[rrr,bend left,densely dotted ,-] & 2 \ar[r] & 3 \ar[r] \ar[rrr,bend left ,densely dotted,-] & 4 \ar[r] \ar[rrr,bend left ,densely dotted,-] & 5 \ar[r] & 6 \ar[r] \ar[rrr,bend left ,densely dotted,-] & 7 \ar[r] & 8 \ar[r] & 9.
\end{tikzcd} \]

Now Corollary~\ref{corollary:introducevertex} tells us that the following Nakayama algebras (and their opposites) are all non-piecewise hereditary.
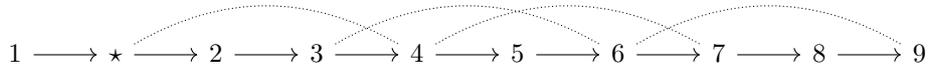
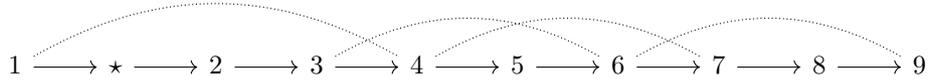
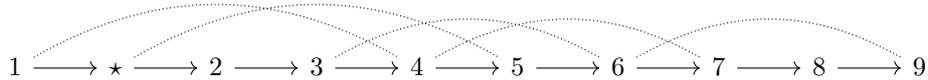
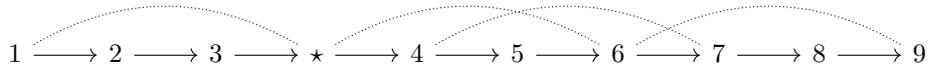
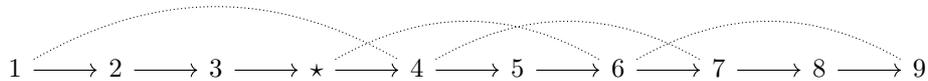
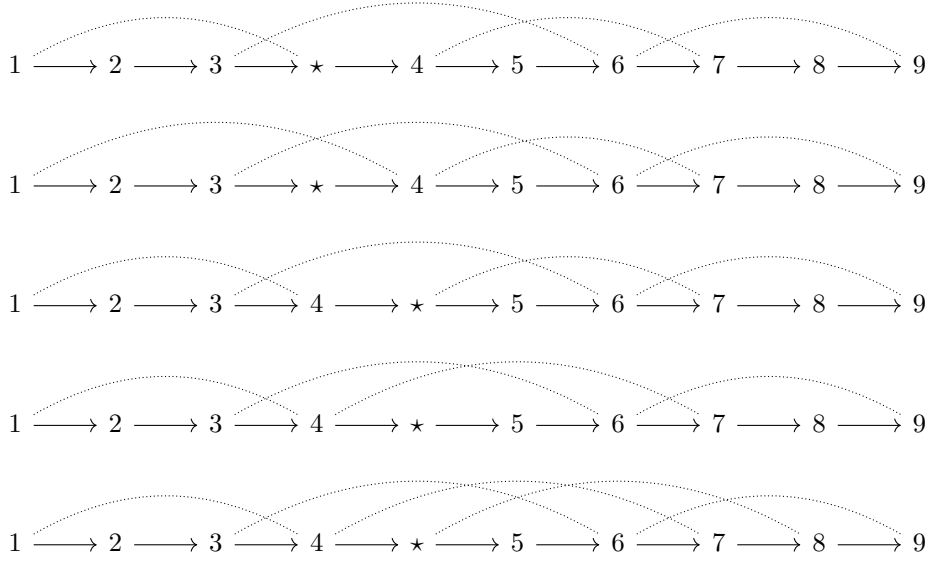
\[ \begin{tikzcd}
1 \ar[r] & \star \ar[r] \ar[rrr,bend left,densely dotted ,-] & 2 \ar[r] & 3 \ar[r] \ar[rrr,bend left,densely dotted ,-] & 4 \ar[r] \ar[rrr,bend left,densely dotted ,-] & 5 \ar[r] & 6 \ar[r] \ar[rrr,bend left,densely dotted ,-] & 7 \ar[r] & 8 \ar[r] & 9
\end{tikzcd} \]
\[ \begin{tikzcd}
1 \ar[r] \ar[rrrr,bend left,densely dotted ,-] & \star \ar[r] & 2 \ar[r] & 3 \ar[r] \ar[rrr,bend left,densely dotted ,-] & 4 \ar[r] \ar[rrr,bend left,densely dotted ,-] & 5 \ar[r] & 6 \ar[r] \ar[rrr,bend left,densely dotted ,-] & 7 \ar[r] & 8 \ar[r] & 9
\end{tikzcd} \]
\[ \begin{tikzcd}
1 \ar[r]\ar[rrrr,bend left,densely dotted ,-] & \star \ar[r] \ar[rrrr,bend left,densely dotted ,-] & 2 \ar[r] & 3 \ar[r] \ar[rrr,bend left,densely dotted ,-] & 4 \ar[r] \ar[rrr,bend left,densely dotted ,-] & 5 \ar[r] & 6 \ar[r] \ar[rrr,bend left,densely dotted ,-] & 7 \ar[r] & 8 \ar[r] & 9
\end{tikzcd} \]
\[ \begin{tikzcd}
1 \ar[r] \ar[rrr,bend left,densely dotted ,-] & 2 \ar[r] & 3 \ar[r] & \star \ar[r] \ar[rrr,bend left,densely dotted ,-] & 4 \ar[r] \ar[rrr,bend left,densely dotted ,-] & 5 \ar[r] & 6 \ar[r] \ar[rrr,bend left,densely dotted ,-] & 7 \ar[r] & 8 \ar[r] & 9
\end{tikzcd} \]
\[ \begin{tikzcd}
1 \ar[r] \ar[rrrr,bend left,densely dotted ,-] & 2 \ar[r] & 3 \ar[r] & \star \ar[r] \ar[rrr,bend left,densely dotted ,-] & 4 \ar[r] \ar[rrr,bend left,densely dotted ,-] & 5 \ar[r] & 6 \ar[r] \ar[rrr,bend left,densely dotted ,-] & 7 \ar[r] & 8 \ar[r] & 9
\end{tikzcd} \]
\[ \begin{tikzcd}
1 \ar[r] \ar[rrr,bend left,densely dotted ,-]& 2 \ar[r] & 3 \ar[r] \ar[rrrr,bend left,densely dotted ,-] & \star \ar[r] & 4 \ar[r] \ar[rrr,bend left,densely dotted ,-] & 5 \ar[r] & 6 \ar[r] \ar[rrr,bend left,densely dotted ,-] & 7 \ar[r] & 8 \ar[r] & 9
\end{tikzcd} \]
\[ \begin{tikzcd}
1 \ar[r] \ar[rrrr,bend left,densely dotted ,-]& 2 \ar[r] & 3 \ar[r] \ar[rrrr,bend left,densely dotted ,-] & \star \ar[r] & 4 \ar[r] \ar[rrr,bend left,densely dotted ,-] & 5 \ar[r] & 6 \ar[r] \ar[rrr,bend left,densely dotted ,-] & 7 \ar[r] & 8 \ar[r] & 9
\end{tikzcd} \]
\[ \begin{tikzcd}
1 \ar[r] \ar[rrr,bend left,densely dotted ,-] & 2 \ar[r] & 3 \ar[r] \ar[rrrr,bend left,densely dotted ,-]& 4 \ar[r] & \star \ar[r] \ar[rrr,bend left,densely dotted ,-] & 5 \ar[r] & 6 \ar[r] \ar[rrr,bend left,densely dotted ,-] & 7 \ar[r] & 8 \ar[r] & 9
\end{tikzcd} \]
\[ \begin{tikzcd}
1 \ar[r] \ar[rrr,bend left,densely dotted ,-] & 2 \ar[r] & 3 \ar[r] \ar[rrrr,bend left,densely dotted ,-]& 4 \ar[r] \ar[rrrr,bend left,densely dotted ,-]& \star \ar[r]  & 5 \ar[r] & 6 \ar[r] \ar[rrr,bend left,densely dotted ,-] & 7 \ar[r] & 8 \ar[r] & 9
\end{tikzcd} \]
\[ \begin{tikzcd}
1 \ar[r] \ar[rrr,bend left,densely dotted ,-] & 2 \ar[r] & 3 \ar[r] \ar[rrrr,bend left,densely dotted ,-]& 4 \ar[r] \ar[rrrr,bend left,densely dotted ,-]& \star \ar[r] \ar[rrrr,bend left,densely dotted ,-] & 5 \ar[r] & 6 \ar[r] \ar[rrr,bend left,densely dotted ,-] & 7 \ar[r] & 8 \ar[r] & 9
\end{tikzcd} \]

\end{example}

\begin{example}\label{example:A13notPWH}
We can extend the algebra from Example~\ref{example:A11tworelations} by introducing a vertex between $5$ and $6$ with a relation from it to $11$, and a vertex between $6$ and $7$ with a relation to it from $1$. The resulting algebra is
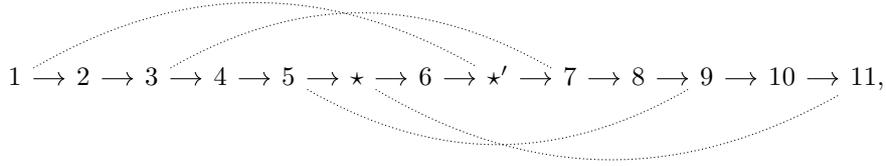
\begin{equation*}
\begin{tikzcd}[column sep = small]
1 \ar[r] \ar[rrrrrrr, bend left, densely dotted, -]& 2 \ar[r] & 3 \ar[r] \ar[rrrrrr, bend left, densely dotted, -] & 4 \ar[r] & 5 \ar[r] \ar[rrrrrr, bend right, densely dotted, -]& \star \ar[r] \ar[rrrrrrr, bend right, densely dotted, -] & 6 \ar[r] & \star' \ar[r] & 7 \ar[r] & 8 \ar[r] & 9 \ar[r] & 10 \ar[r] & 11,
\end{tikzcd}
\end{equation*}
which is still not piecewise hereditary by Corollary~\ref{corollary:introducevertex}. This algebra is isomorphic to $\Lambda$ from Example~\ref{example:A13tworelations}, which as we saw there is derived equivalent to $\Lambda'$ given by
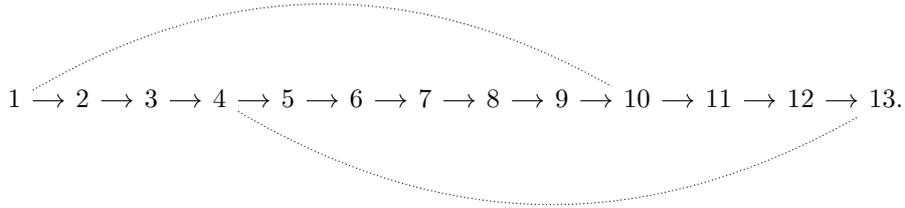
\begin{equation*}
\begin{tikzcd}[column sep = small]
1 \ar[r] \ar[rrrrrrrrr, bend left, densely dotted, -] & 2 \ar[r] & 3 \ar[r] & 4 \ar[r] \ar[rrrrrrrrr, bend right, densely dotted, -]& 5 \ar[r] & 6 \ar[r] & 7 \ar[r] & 8 \ar[r] & 9 \ar[r] & 10 \ar[r] & 11 \ar[r] & 12 \ar[r] & 13.
\end{tikzcd}
\end{equation*}
Hence, $\Lambda'$ is not piecewise hereditary.
\end{example}

\subsection*{Application: Radical powers}
Here we recover the part of Table $1$ in \cite{MR2736030} that consists of algebras which are not piecewise hereditary. In the notation of that paper, $\Lambda(n,r)$ is the Nakayama algebra $k\mathbb{A}_n/(\rad)^r$.

\begin{proposition}[\!\!\text{\cite[Proposition~2.6]{MR2736030}}]\label{proposition:HS}
The following algebras are not piecewise hereditary.
\begin{itemize}
\item $\Lambda(n+i, n-5)$ for $n\geq 12$ and $i\geq 0$.
\item $\Lambda(n,3)$ for $n\geq 12$.
\item $\Lambda(n,4)$ for $n\geq 11$.
\item $\Lambda(n,5)$ for $n\geq 11$.
\item $\Lambda(n,6)$ for $n\geq 12$.
\end{itemize}
\end{proposition}

\begin{remark}
The following proof is different from that of \cite{MR2736030}. In particular, here we neither assume that $k$ is algebraically closed nor appeal to a classification of hereditary categories.
\end{remark}

\begin{proof}
The proposition claims precisely that the marked entries indicated by the following table, are all non-piecewise hereditary.
\[\begin{tabular}{l*{7}{c}}
\diagbox[width=2em, height=2em]{\tiny{$r$}}{\tiny{$n$}} &11&12&13&14&15&16&$\cdots$\\
$3$ &&\tiny{$\otimes$}&\tiny{$\times$}&\tiny{$\times$}&\tiny{$\times$}&\tiny{$\times$}&\tiny{$\cdots$}\\
$4$&\tiny{$\otimes$}&\tiny{$\times$}&\tiny{$\times$}&\tiny{$\times$}&\tiny{$\times$}&\tiny{$\times$}&\tiny{$\cdots$}\\
$5$&\tiny{$\otimes$}&\tiny{$\times$}&\tiny{$\times$}&\tiny{$\times$}&\tiny{$\times$}&\tiny{$\times$}&\tiny{$\cdots$}\\
$6$&&\tiny{$\otimes$}&\tiny{$\times$}&\tiny{$\times$}&\tiny{$\times$}&\tiny{$\times$}&\tiny{$\cdots$}\\
$7$&&\tiny{$\otimes$}&\tiny{$\times$}&\tiny{$\times$}&\tiny{$\times$}&\tiny{$\times$}&\tiny{$\cdots$}\\
$8$&&&\tiny{$\times$}&\tiny{$\times$}&\tiny{$\times$}&\tiny{$\times$}&\tiny{$\cdots$}\\
$9$&&&&\tiny{$\times$}&\tiny{$\times$}&\tiny{$\times$}&\tiny{$\cdots$}\\
$\vdots$ &&&&&\tiny{$\ddots$}&\tiny{$\ddots$}&\tiny{$\ddots$}\\
\end{tabular}\]
Here the five algebras marked with $\otimes$ are known to be not piecewise hereditary from List~\ref{liste:baseforHS}. We will complete the table inductively by appealing to Corollary~\ref{corollary:introducevertex}.

First, from the algebra $\Lambda(12,7)$, namely
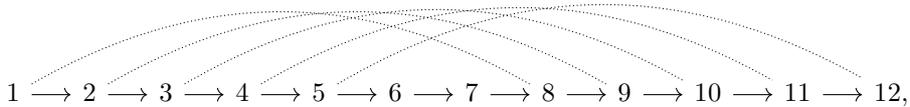
\[ \begin{tikzcd}[column sep=1.5em]
1 \ar[r] \ar[rrrrrrr,bend left,densely dotted,-] & 2 \ar[r] \ar[rrrrrrr,bend left,densely dotted,-] & 3 \ar[r] \ar[rrrrrrr,bend left,densely dotted,-] & 4 \ar[r] \ar[rrrrrrr,bend left,densely dotted,-] & 5 \ar[r] \ar[rrrrrrr,bend left,densely dotted,-] & 6 \ar[r] & 7 \ar[r] & 8 \ar[r] & 9 \ar[r] & 10 \ar[r] & 11 \ar[r] & 12,
\end{tikzcd}\]
one obtains $\Lambda(13,8)$ by introducing a ``central" vertex:
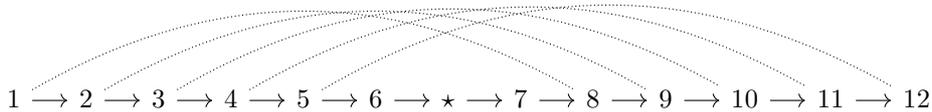
\[ \begin{tikzcd}[column sep=1.35em]
1 \ar[r] \ar[rrrrrrrr,bend left,densely dotted,-] & 2 \ar[r] \ar[rrrrrrrr,bend left,densely dotted,-] & 3 \ar[r] \ar[rrrrrrrr,bend left,densely dotted,-] & 4 \ar[r] \ar[rrrrrrrr,bend left,densely dotted,-] & 5 \ar[r] \ar[rrrrrrrr,bend left,densely dotted,-] & 6 \ar[r] & \star \ar[r] & 7 \ar[r] & 8 \ar[r] & 9 \ar[r] & 10 \ar[r] & 11 \ar[r] & 12
\end{tikzcd}\]
By Corollary~\ref{corollary:introducevertex} the latter algebra is not piecewise hereditary. This procedure may be iterated, clearly, and it follows that each algebra on the lowest diagonal of the table is not piecewise hereditary.

Now the table can be completed by a horizontal argument: The algebra $\Lambda(n+1,r)$ is obtained from $\Lambda(n,r)$ by adding a vertex and a relation of length $r$ at the ``start" of the quiver. For instance, in this way $\Lambda(11,5)$, namely
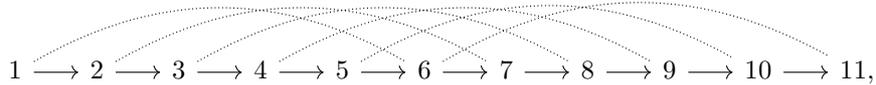
\[ \begin{tikzcd}[column sep=1.7em]
1 \ar[r] \ar[rrrrr,bend left,densely dotted,-] & 2 \ar[r] \ar[rrrrr,bend left,densely dotted,-] & 3 \ar[r] \ar[rrrrr,bend left,densely dotted,-] & 4 \ar[r] \ar[rrrrr,bend left,densely dotted,-] & 5 \ar[r] \ar[rrrrr,bend left,densely dotted,-] & 6 \ar[r] \ar[rrrrr,bend left,densely dotted,-] & 7 \ar[r] & 8 \ar[r] & 9 \ar[r] & 10 \ar[r] & 11,
\end{tikzcd}\]
gives rise to $\Lambda(12,5)$:
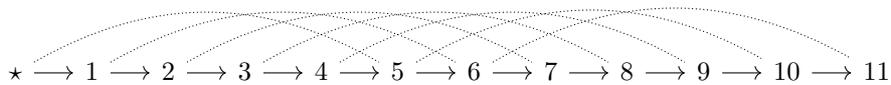
\[ \begin{tikzcd}[column sep=1.5em]
\star \ar[r] \ar[rrrrr,bend left,densely dotted,-] & 1 \ar[r] \ar[rrrrr,bend left,densely dotted,-] & 2 \ar[r] \ar[rrrrr,bend left,densely dotted,-] & 3 \ar[r] \ar[rrrrr,bend left,densely dotted,-] & 4 \ar[r] \ar[rrrrr,bend left,densely dotted,-] & 5 \ar[r] \ar[rrrrr,bend left,densely dotted,-] & 6 \ar[r] \ar[rrrrr,bend left,densely dotted,-] & 7 \ar[r] & 8 \ar[r] & 9 \ar[r] & 10 \ar[r] & 11
\end{tikzcd}\]
By Corollary~\ref{corollary:introducevertexatend} the assertion that $\Lambda(n,r)$ is not piecewise hereditary implies that $\Lambda(n+1,r)$ is also not piecewise hereditary, which completes the table.
\end{proof}

\section{Families of non-piecewise hereditary algebras}

In this last section we provide two simple criteria for Nakayama algebras to be non-piecewise hereditary. The cases covered here make up a large number of such algebras, but are far from a complete list.

In spirit, the propositions below say that if a Nakayama algebra contains sufficiently overlapping relations then it is not piecewise hereditary. As such they give a partial converse to \cite[Theorem 4.2]{Fosse2023}, which states that if no pair of relations overlap by more than one arrow then the algebra \emph{is} piecewise hereditary (by virtue of being derived equivalent to the path algebra of a tree).

\begin{proposition}\label{proposition:A13tworelations}
Let \( \Lambda \) be a Nakayama algebra with relations $\alpha$ and $\beta$ that overlap by at least six arrows ($\alpha$ and $\beta$ do not have to be consecutive). Let $\alpha$ be the leftmost of the two, and assume additionally that
\begin{itemize}
\item there are at least two vertices between the start of $\alpha$ and the start of $\beta$, and at least two vertices between the end of $\alpha$ and the end of $\beta$, and
\item there is no relation going from the first or second vertex before the start of $\beta$ to the first or second vertex after the end of $\alpha$.
\end{itemize}

Then $\Lambda$ is not piecewise hereditary.

\end{proposition}

\begin{proof}
We will prove the proposition by showing that any such algebra can be constructed by attaching --- by way of Corollaries~\ref{corollary:introducevertex} and \ref{corollary:introducevertexatend} --- some set of vertices and relations to an algebra we already know is non-piecewise hereditary. We start with $\Lambda'$ from Example~\ref{example:A13notPWH}:
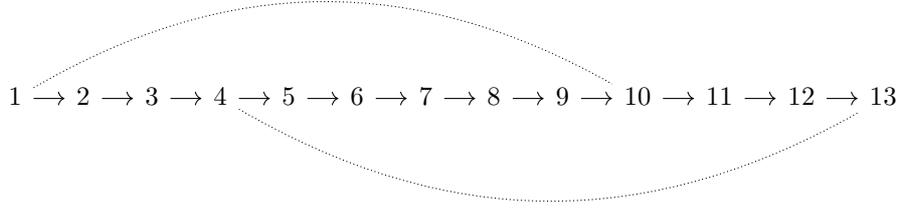
\begin{equation*}
\begin{tikzcd}[column sep = small]
1 \ar[r] \ar[rrrrrrrrr, bend left, densely dotted, -] & 2 \ar[r] & 3 \ar[r] & 4 \ar[r] \ar[rrrrrrrrr, bend right, densely dotted, -]& 5 \ar[r] & 6 \ar[r] & 7 \ar[r] & 8 \ar[r] & 9 \ar[r] & 10 \ar[r] & 11 \ar[r] & 12 \ar[r] & 13
\end{tikzcd}
\end{equation*}

By Corollary~\ref{corollary:introducevertexatend} it suffices to consider the case when $\alpha$ begins at the start of the quiver, and $\beta$ ends at the end of the quiver. In this case the quiver with relations for \( \Lambda \) has the form
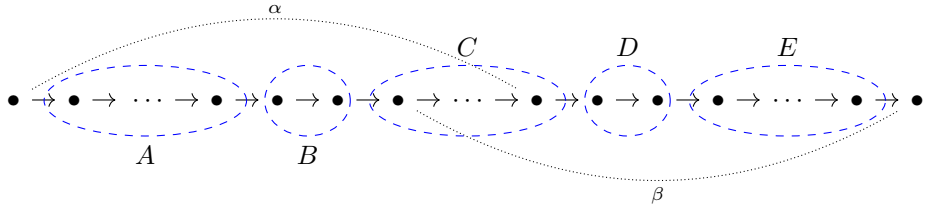
\begin{equation*}
\begin{tikzcd}[column sep = .9em, execute at end picture={
     \node[dashellipse=(\tikzcdmatrixname-1-2)(\tikzcdmatrixname-1-3)(\tikzcdmatrixname-1-4), label=below:{$A$}]{};
     \node[dashellipse=(\tikzcdmatrixname-1-5)(\tikzcdmatrixname-1-6), label=below:{$B$}]{};
     \node[dashellipse=(\tikzcdmatrixname-1-7)(\tikzcdmatrixname-1-8)(\tikzcdmatrixname-1-9), label=above:{$C$}]{};
     \node[dashellipse=(\tikzcdmatrixname-1-10)(\tikzcdmatrixname-1-11), label=above:{$D$}]{};
     \node[dashellipse=(\tikzcdmatrixname-1-12)(\tikzcdmatrixname-1-13)(\tikzcdmatrixname-1-14), label=above:{$E$}]{};
}]
\bullet \ar[r] \ar[rrrrrrrr, densely dotted, bend left, -, "\alpha"] & \bullet \ar[r] & {} \cdots \ar[r] & \bullet \ar[r] & \bullet \ar[r] & \bullet \ar[r] & \bullet \ar[r] \ar[rrrrrrrr, densely dotted, bend right, -,swap,"\beta"] & \cdots \ar[r] & \bullet \ar[r] & \bullet \ar[r] & \bullet \ar[r] & \bullet \ar[r] & \cdots \ar[r] & \bullet \ar[r] & \bullet
\end{tikzcd}
\end{equation*}
with the following properties.
\begin{enumerate}[label=$\Alph*$:]
\item Contains zero or more vertices, each of which may have a relation to a vertex in $D\cup E$.
\item Contains exactly two vertices, each of which may have a relation to a vertex in $E$.
\item Contains at least seven vertices. There is the relation \( \alpha \) from the first vertex of the algebra to the last vertex of $C$, and the relation \( \beta \) from the first vertex of $C$ to the last vertex of the algebra, but no other relations starting or ending in $C$.
\item Contains exactly two vertices, each of which may be hit by a relation from a vertex in $A$.
\item Contains zero or more vertices, each of which may be hit by a relation from a vertex in $A \cup B$.
\end{enumerate}

We can obtain this picture from the previous one by introducing new vertices as in Corollary~\ref{corollary:introducevertex}:

We can introduce additional vertices in \( C \) without problem, since they do not involve any new relations.

For the vertices in \( A \) we proceed from right to left. At any point, the new vertex introduced will have the relation \( \alpha \) starting at its immediate predecessor, so by Corollary~\ref{corollary:introducevertex} we may introduce a relation starting in this new vertex. Dually, we introduce the vertices in \( E \) iteratively from left to right, such that each new vertex is the immediate predecessor of the end of \( \beta \) at the time of its introduction.
\end{proof}

\begin{example}
Let $\Lambda$ be given by
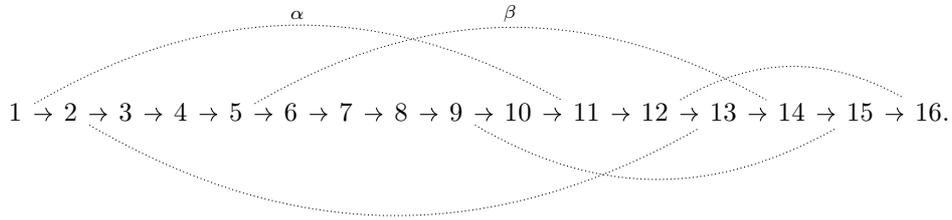
\[ \begin{tikzcd}[column sep = 0.7em]
1 \ar[r] \ar[rrrrrrrrrr,bend left,densely dotted,-,"\alpha"] & 2 \ar[r] \ar[rrrrrrrrrrr,bend right,densely dotted,-] & 3 \ar[r] & 4 \ar[r] & 5 \ar[r] \ar[rrrrrrrrr,bend left,densely dotted,-,"\beta"] & 6 \ar[r] & 7 \ar[r] & 8 \ar[r] & 9 \ar[r] \ar[rrrrrr, bend right, densely dotted, -] & 10 \ar[r] & 11 \ar[r] & 12 \ar[r] \ar[rrrr,bend left,densely dotted,-] & 13 \ar[r] & 14 \ar[r] & 15 \ar[r] & 16.
\end{tikzcd} \]
Here the relations $\alpha$ and $\beta$ satisfy the conditions in Proposition~\ref{proposition:A13tworelations}, which implies that $\Lambda$ is not piecewise hereditary.
\end{example}

\bigskip

By expanding on Example~\ref{example:A13notPWH} we got Proposition~\ref{proposition:A13tworelations}. In a similar vein, also the quiver with relations from Example~\ref{example:A9} will yield some general patterns that represent non-piecewise hereditary Nakayama algebras. The criteria obtained in this way may be summarized as follows.

\begin{proposition}\label{proposition:A9extended}
Let \( \Lambda \) be a Nakayama algebra whose ordinary quiver has at least 9 vertices, and relations $\alpha$ and $\beta$ that overlap by at least two arrows ($\alpha$ and $\beta$ need not be consecutive). Let $\alpha$ be the leftmost of the two, and assume additionally:
\begin{enumerate}
\item\label{item:prop1} There \textbf{is no} relation starting in the vertex directly before the first vertex of $\alpha$ and ending at or before the vertex directly after the first vertex of $\beta$.
\item\label{item:prop2} There \textbf{is no} relation ending in the vertex directly after the last vertex of $\beta$ and starting at or after the vertex directly before the last vertex of $\alpha$.
\item\label{item:prop3} There \textbf{is} a relation of length $\ge3$ which lies entirely before $\beta$ (possibly sharing a vertex).
\item\label{item:prop4} There \textbf{is} a relation of length $\ge3$ which lies entirely after $\alpha$ (possibly sharing a vertex).
\end{enumerate}

Then \( \Lambda \) is not piecewise hereditary.
\end{proposition}

\begin{proof}
By Corollary~\ref{corollary:lengthtworelations} we may assume that there are no relations of length two, and by Corollary~\ref{corollary:introducevertexatend} we may assume that there is a single relation each as in (3) and (4), and that these relations start at the very start of the quiver and end at the very end of the quiver, respectively.

Recall from Example~\ref{example:A9} that the algebra defined by
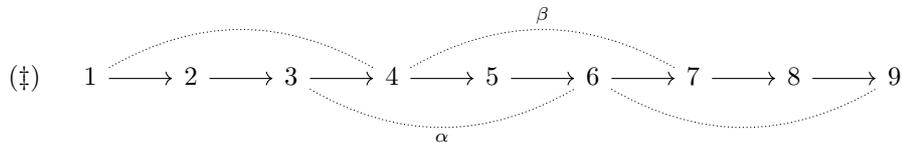
\begin{equation}\label{eq:exA9againagain}
\begin{tikzcd}
1 \ar[r] \ar[rrr,bend left,densely dotted,-] & 2 \ar[r] & 3 \ar[r] \ar[rrr,bend right,densely dotted,-,swap,"\alpha"] & 4 \ar[r] \ar[rrr,bend left,densely dotted,-,"\beta"] & 5 \ar[r] & 6 \ar[r] \ar[rrr,bend right,densely dotted,-] & 7 \ar[r] & 8 \ar[r] & 9
\end{tikzcd}\tag{$\ddagger$}
\end{equation}
is not piecewise hereditary. Our strategy now is to repeatedly apply Corollary~\ref{corollary:introducevertex}; it follows that any quiver with relations obtained by driving (\ref{eq:exA9againagain}) through the following procedure, represents a non-piecewise hereditary Nakayama algebra.

Firstly we can add a vertex between \( 4 \) and \( 5 \), optionally along with a relation starting in the new vertex and ending after the last vertex of \( \beta \). Likewise, we may add a vertex between \( 5 \) and \( 6 \), with a relation starting before the first vertex of \( \alpha \) and ending in the new vertex.

Secondly, adding a vertex between \( 1 \) and \( 2 \) increases the length of the leftmost relation. Note that this new vertex is allowed to have a relation starting in it (which will then end at any vertex between the start of \( \beta \) and the end of \( \alpha \)). This construction can be done repeatedly, increasing the length of the leftmost relation even more. Symmetrically, we can increase the length of the rightmost relation.

Thirdly we can introduce a new vertex between \( 3 \) and \( 4 \), making this new vertex the end-point of the leftmost relation as well as the start-point of \( \alpha \). Effectively this moves the leftmost relation one step left, and iterating this procedure we can move it as far left as we like relative to the start of \( \alpha \). Symmetrically, we can move the rightmost relation as far to the right as we like, compared to the end of \( \beta \) (extending the quiver as necessary).

Finally we can again introduce additional vertices between \( 3 \) and \( 4 \), this time without making these new vertices the start-points of  \( \alpha \), thus creating an arbitrary (positive) distance between the starts of \( \alpha \) and \( \beta \). When also introducing symmetric vertices between \( 6 \) and \( 7 \) we are additionally allowed to introduce relations between these two sets of newly created vertices.

It now suffices to observe that these four steps allow us to create any quiver with relations satisfying the assumptions of the proposition, which additionally do not contain relations of length two or unnecessary parts at the ends as in the first sentence of this proof.
\end{proof}

\begin{corollary}\label{corollary:NonOverlappingRelSets}
Let $\Lambda$ be a Nakayama algebra. Suppose that the ordinary quiver of $\Lambda$ contains three non-empty sets of mutually non-overlapping relations, where the middle set contains a pair of relations whose intersection is at least two arrows.

Then $\Lambda$ is not piecewise hereditary.
\end{corollary}

\begin{example}
By Corollary~\ref{corollary:NonOverlappingRelSets} the algebra given by
\[ \begin{tikzcd}[column sep = 0.7em]
1 \ar[r] \ar[rrrr,bend left,densely dotted,-] & 2 \ar[r] & 3 \ar[r] & 4 \ar[r] & 5 \ar[r] & 6 \ar[r] \ar[rrrr,bend right,densely dotted,-] & 7 \ar[r] & 8 \ar[r] \ar[rrr,bend left,densely dotted,-] & 9 \ar[r] & 10 \ar[r] & 11 \ar[r] & 12 \ar[r] \ar[rrr,bend right,densely dotted,-] & 13 \ar[r] \ar[rrr, bend left, densely dotted, -]& 14 \ar[r] & 15 \ar[r] & 16
\end{tikzcd} \]
is not piecewise hereditary.
\end{example}

\bibliographystyle{amsplain}
\bibliography{main}
\end{document}